\documentclass{amsart} 
\usepackage{amsmath}
\usepackage{paralist}
\usepackage[misc]{ifsym}
\usepackage{epsfig} 
\usepackage{epstopdf} 
\usepackage[colorlinks=true]{hyperref}

\usepackage{mathtools}
\mathtoolsset{showonlyrefs=true}

\hypersetup{urlcolor=blue, citecolor=red}
\allowdisplaybreaks

\usepackage{amsmath,amssymb,mathrsfs}

\def\R{\mathbb{R}}

\def\R{\mathbf{R}}
\def\highlightEdit{\blue}

\subjclass{Primary: 35Q83, 82C40.}
\keywords{Local well-posedness, Vlasov-Poisson-Landau system, kinetic theory, massless electron limit, massless electron system, Landau equation.}
 
 \def\blue{}
\newtheorem{theorem}{Theorem}[section]

\newtheorem{lemma}[theorem]{Lemma}
\newtheorem{proposition}[theorem]{Proposition}

\theoremstyle{definition}

\newtheorem{remark}[theorem]{Remark}

\newcommand{\ep}{\varepsilon}

\title[Local well-posedness for the VPL system]{Local Well-Posedness of the Vlasov-Poisson-Landau System and Related Models}

\thanks{
Patrick Flynn was partially supported by the National Science Foundation Graduate Research Fellowship, Grant No. 2040433. Thank you to Yan Guo, Beno\^it Pausader and Timur Yastrzhembskiy for helpful conversations and suggestions.}

\begin{document}
\maketitle
\centerline{\scshape
{Patrick Flynn$^{{\href{mailto:pflynn@math.ucla.edu}{\textrm{\Letter}}}*1}$}}
\medskip

{\footnotesize
 \centerline{$^1$University of California, Los Angeles, United States of America}
} 

\bigskip

\begin{abstract}
We prove local well-posedness for the Vlasov-Poisson-Landau system and the variant with massless electrons in a 3D periodic spatial domain for large initial data. This is accomplished by propagating weighted anisotropic $L^2$-based Sobolev norms. In the case of the massless electron system, we also carry out an analysis of the Poincar\'e-Poisson system. This is a companion paper to the author's previous work with Yan Guo \cite{flynn2024massless}.
\end{abstract}

\section{Introduction} 

Consider a plasma whose ion and electron distributions are given by $F_+(t,x,v)$, $F_-(t,x,v)$, respectively, which satisfy the two-species Vlasov-Poisson-Landau system,
\begin{equation}\label{eq:VPLunit}
\begin{aligned}
\{\partial_t  +v \cdot \nabla_x  +E \cdot \nabla_v\} F_+ &= Q(F_+ + F_-, F_+),  \\
\{ \partial_t  + v \cdot \nabla_x  - E\cdot \nabla_{\highlightEdit{v}} \}  F_- &=Q(F_++F_-, F_-), \\
 - \Delta_x  \phi &=  4\pi ( n_+ - n_-).
\end{aligned}
\end{equation}
 Above, $\phi$ and $E = -\nabla_x \phi$ are the electric potential and field respectively. The functions $n_\pm(t,x)$ are the charge densities, defined by
\begin{align}
n_+ &= \int_{\mathbf R^3} F_+ dv , \ \ \ n_- =\int_{\mathbf R^3} F_- dv.
\end{align}
 The collision operator is given by
\begin{align}
Q(G_1,G_2) &= \nabla_v \cdot\int_{\mathbf R^3} \Phi(v- v') \left\{G_1(v') \nabla_v G_2(v) - G_2(v) \nabla_{v'} G_1 (v')\right\}dv'.
\end{align}
 \highlightEdit{Above, $\Phi$ is the Landau collision kernel associated with Coulombic interactions, defined as follows:} for each $z \in \mathbf R^3$, 
\begin{equation}
\Phi(z) := \frac{1}{|z|}\left (I - \frac{z\otimes z}{|z|^2}\right).
\end{equation}
We  note that we have set all physical parameters to 1 \highlightEdit{for the sake of simplicity. However, we emphasize that the results of this paper apply to general choices of parameters (see Remark \ref{rem:parameters} below).}
Finally, we chose the spatial domain to be the 3D torus, i.e. $x \in \mathbf T^3 = (\mathbf R /\mathbf Z)^3$ and $v \in \mathbf R^3$.

 In the paper \cite{flynn2024massless}, the Vlasov-Poisson-Landau system \highlightEdit{with massless electrons} was derived from the system \eqref{eq:VPLunit} (albeit, with the electron mass set to a small number, rather than 1).  \highlightEdit{This system} involves the ion distribution function $F_+(t,x,v)$ coupled to the inverse electron temperature $\beta(t)$ (a time-dependent scalar quantity) and the electrostatic field $\phi(t,x)$.
 The system for $(F_+,\beta,\phi)$ reads 
\begin{equation}\label{eq:VPLion''}
\begin{aligned}
&\{\partial_t  +v \cdot \nabla_x  - E \cdot \nabla_v\} F_+ = Q(F_+, F_+)\\
&\frac{d}{dt} \left\{\frac{3}{2\beta} + \iint_{\mathbf T^3 \times \mathbf R^3} \frac{|v|^2}{2} F_+dx dv + \frac{1}{8\pi}  \int_{\mathbf T^3} |E|^2 dx \right\} = 0,\\
&-\Delta_x  \phi = 4\pi (n_+ - e^{\beta  \phi}).
\end{aligned}
\end{equation}
This system has been formally derived in a number of papers prior to this work, namely \cite{bardos_maxwellboltzmann_2018} in the Boltzmann case, or  \cite{degond1996asymptotics,degond1996transport} with Landau collisions. \highlightEdit{There are a number of works on related other kinetic equations with massless electrons, such as the Vlasov-Poisson system with massless electrons {\cite{bouchut1991global,gagnebin2023landau,griffin2021global,han2011quasineutral,huang2022nonlinear}}, and other models \cite{herda2016massless}.}

The main result in \cite{flynn2024massless} is conditioned on the existence of sufficiently regular solutions to the limiting system \eqref{eq:VPLion''}. 
Even for the system \eqref{eq:VPLunit}, it appears that a local well-posedness result for large data is missing from the literature, \highlightEdit{although there are several works which answer a similar questions.} The works \cite{he2014well, henderson2019local, henderson2020local} prove such a result in the case of the inhomogeneous Landau system. 
 Also worth mentioning is the local well-posedness result \cite{chaturvedi2019local} in the case of hard potentials. However, it is not immediately clear that any of these works imply an analogous result for the systems considered here. On the other hand, the work \cite{guo_global_2010} proves global well-posedness of the system \eqref{eq:VPLunit} for perturbations of Maxwellian. \highlightEdit{Very recently, \cite{guillen2023landau} proved that the homogeneous Landau equation enjoys global well-posedness.}

 The goal of this paper is to prove local well-posedness for the systems \eqref{eq:VPLunit} and \eqref{eq:VPLion''} with a periodic spatial domain in $L^2$ based weighted anisotropic Sobolev spaces for large initial data. \highlightEdit{ We emphasize that the solutions of the ion system \eqref{eq:VPLion''} constructed in Theorem \ref{thm:LWP}-(i) below are compatible with the hypothesis of the main theorem in \cite{flynn2024massless}, allowing us to rigorously take the massless electron limit (see Remark \ref{rem:compatibility} for details).}


\subsection{Main Result}
Before we can state our main theorem, we must define a number of function spaces in which the solutions shall be constructed. The framework used here builds on the techniques developed in \cite{flynn2024massless}. 
First we define the diffusion matrix associated with maxwellians: let $\sigma_{ij}(v) = \Phi_{ij} * \mu$, with $i, j \in \{1,2,3\}$. Given $u\in \mathbf R^3$, we have that
\begin{equation}
u^T \sigma(v)u  \sim \frac{1}{\langle v\rangle^3} |P_v u|^2 +  \frac{1}{\langle v\rangle}|P_{v^\perp} u|^2\highlightEdit{.}\end{equation}
Here, $P_v + P_{v^\perp} = I$ denote the orthogonal projections onto $\mathrm{span}\{v\}$ and $\{v\}^\perp$ respectively.
Using \highlightEdit{these matrices}, we define the following spaces which will capture the dissipation produced by the various collision operators.
\highlightEdit{We define the $\dot {\mathcal H}_{\sigma}$ semi-norm: for any appropriate $\psi : \R^3\to \R$, let}
\begin{equation}
\| \psi\|_{\dot{\mathcal H}_\sigma} ^2:= \int_{\mathbf R^3} \sigma_{ij}(v)\partial_{v_i} \psi(v)\partial_{v_j}\psi(v)dv.
\end{equation}
\highlightEdit{
Above and throughout this paper, we sum over repeated indices.}

Fix  parameters $s\in (\frac{5}{2},3]$,  $r \in (0,1]$, and $m_j =5(j+1)$ with $j \in \{0,1,2\}$.  Given $u = u(x,v)$, we define
\begin{align}
\|u\|_{\mathfrak E}^2 &:= \|\langle v\rangle^{m_2} u\|_{L^2_{x,v}}^2 + \|\langle v\rangle^{m_1}\langle \nabla_x\rangle^s  u\|_{L^2_{x,v}}^2 + \|\langle \nabla_v\rangle^r u\|_{L^2_{x,v}}^2, \\
\|u\|_{\mathfrak D}^2 &:= \|\langle v\rangle^{m_2} u\|_{L^2_{x}(\dot{\mathcal H}_\sigma)_v}^2 + \|\langle v\rangle^{m_1}\langle \nabla_x\rangle^s  u\|_{L^2_{x}(\dot{\mathcal H}_\sigma)_v}^2  + \|\langle \nabla_v\rangle^r u\|_{L^2_{x}(\dot{\mathcal H}_\sigma)_v}^2,\\
\|u\|_{\mathfrak E'}^2 &:= \|\langle v\rangle^{m_1} u\|_{L^2_{x,v}}^2 + \|\langle v\rangle^{m_0}\langle \nabla_x\rangle^s  u\|_{L^2_{x,v}}^2, \\
\|u\|_{\mathfrak D'}^2 &:= \|\langle v\rangle^{m_1} u\|_{L^2_{x}(\dot{\mathcal H}_\sigma)_v}^2 + \|\langle v\rangle^{m_0}\langle \nabla_x\rangle^s  u\|_{L^2_{x}(\dot{\mathcal H}_\sigma)_v}^2. 
\end{align}
We also define $\mathfrak E_T = L^\infty([0,T];\mathfrak E)$, $\mathfrak D_T = L^2([0,T];\mathfrak D)$ and similarly for the primed spaces.

We now state the theorem on local well-posedness for the systems \eqref{eq:VPLunit} and \eqref{eq:VPLion''}.
\begin{theorem}\label{thm:LWP} Let $F_{+,in}(x,v), F_{-,in}(x,w) $ be nonnegative functions in $L^1(\mathbf T^3\times \mathbf R^3)$ with unit $L^1$ norm.\highlightEdit{
 \begin{enumerate}
\item[(i)] 
Assume $\beta_{in} > 0$ and 
\begin{align}
 \|F_{+,in}\|_{\mathfrak E} + \|\frac{1}{n_{+,in}}\|_{L^\infty} < \infty
\end{align}
Then, there exists $T^* > 0$, such that there exists a unique solution $(F_+,\beta,\phi)$ to \eqref{eq:VPLion''} with initial data $(F_+,\beta)|_{t=0} = (F_{+,in},\beta_{in})$ on $t \in [0,T^*)$ 
\begin{align}
F_+ &\in C([0,T^*);\mathfrak E') \cap L_{loc}^\infty([0,T^*);\mathfrak E) \cap L^2_{loc}([0,T^*); \mathfrak D), \label{eq:F_+regularity}\\
\beta &\in C([0,T^*); \mathbf R_+), \\
\phi &\in C([0,T^*);H^{s+2}).
\end{align}
Moreover, either $T^* = \infty$ or 
\begin{align}
\lim_{T\uparrow T^*} \|F_+\|_{\mathfrak E_T} + \|\frac{1}{n_+}\|_{L^\infty([0,T]\times\mathbf T^3)} = \infty.
\end{align}
\item[(ii)] 
Assume
\begin{align}
 \|(F_{+,in},F_{-,in})\|_{\mathfrak E} + \|(\frac{1}{n_{+,in}},\frac{1}{n_{-,in}})\|_{L^\infty} < \infty
\end{align}
Then, there exists $T^* > 0$, such that there exists a unique solution $(F_+,F_-)$ to \eqref{eq:VPLunit} with initial data $(F_+,F_-)|_{t=0} = (F_{+,in},F_{-,in})$ on $t \in [0,T^*)$  in the space 
\begin{align}
F_\pm &\in C([0,T^*);\mathfrak E') \cap L^\infty_{loc} ([0,T^*);\mathfrak E) \cap L^2_{loc}([0,T^*); \mathfrak D),
\end{align}
Moreover, either $T^* = \infty$ or 
\begin{align}\label{eq:blowup}
\lim_{T \uparrow T^*} \|(F_+,F_-)\|_{\mathfrak E_T} + \|(\frac{1}{n_+},\frac{1}{n_-})\|_{L^\infty([0,T]\times\mathbf T^3)} = \infty.
\end{align}
\end{enumerate}}
\end{theorem}\highlightEdit{
In the theorem above, by ``solution," we mean a weak solution in the sense of distributional derivatives. 
 More precisely, in the case of \eqref{eq:VPLion''}, for any $\varphi \in C^1_{c}([0,T^*)\times \mathbf T^3 \times \mathbf R^3)$, we have
\begin{align*}
 &\iiint_{[0,T^*)\times \mathbf T^3 \times \mathbf R^3}\{\partial_t  + v\cdot \nabla_x - E\cdot \nabla_v\}\varphi(t,x,v) F_+(t,x,v) dx dv\\
 & = \iiiint_{[0,T^*)\times \mathbf T^3 \times \mathbf R^3\times \mathbf R^3}\partial_{v_i} \varphi(t,x,v)\Phi_{ij}(v- v') \\
&\quad \quad \cdot \left\{F_+(t,x,v') \partial_{v_j} F_+(t,x,v) - F_+(t,x,v) \partial_{v_j'} F_+ (t,x,v')\right\}dtdxdvdv'\\
 & \quad - \iiint_{ \mathbf T^3 \times \mathbf R^3}\varphi(0,x,v) F_{+,in}(x,v) dx dv
\end{align*}
One can check that if \eqref{eq:F_+regularity} holds, the nonlinear collision integral above is well-defined using Lemma \ref{lem:PhiUpperLower}. We define solutions similarly for \eqref{eq:VPLunit}.}

\begin{remark}
\label{rem:compatibility} 
The spaces $\mathfrak E$ differ from the spaces in \cite{flynn2024massless} denoted by the same symbols. Specifically, the $\mathfrak E$ norm here controls a small amount of $v$ regularity through the $\langle \nabla_v\rangle^r$ term.  Such control on velocity derivatives was not present in the previous paper, and requires additional analysis. This additional control is necessary to handle the non-perturbative setting of this paper.
It is interesting and important to note that Sobolev regularity in velocity plays different roles in these two papers: specifically, the uniform estimates in \cite{flynn2024massless} are easier to obtain without estimating velocity derivatives due to the singular nature of the ion-electron collision operator as the electron mass tends to zero. On the other hand, the additional $v$ regularity is more convenient to include for LWP of the limiting system \eqref{eq:VPLion''} in this paper.
Nevertheless, since the norms here are stronger, the LWP theorem proved here implies that one can construct a solution to \eqref{eq:VPLion''} compatible with the assumptions of Theorem 1.1 of \cite{flynn2024massless}. 

\end{remark}

\begin{remark}\label{rem:parameters}
In this paper, we have set the atomic number, electron charge, electron mass, proton mass, Coulomb logarithm, and temperature all to 1, for the sake of simplifying the presentation.
The models \eqref{eq:VPLunit} and \eqref{eq:VPLion''} with correct physical parameters can  be found in \cite{flynn2024massless}. Moreover, in that paper, there is a re-scaling of the system that only depends on three dimensionless parameters, the ratio of the electron over ion mass $\ep^2>0$,  the Knudsen number $\kappa >0$, and the atomic number $Z$.

With no serious alterations to the proof, Theorem \ref{thm:LWP} proves local well-posedness of the systems \eqref{eq:VPLunit} and \eqref{eq:VPLion''} with arbitrary positive choices of physical parameters. However, the maximal time of existence $T^*>0$ may go to zero in certain limiting regimes. For instance, the massless electron limit $\ep \to 0$ of the Vlasov-Poisson-Landau system  is singular, and so it is unlikely that one could show a uniform lower bound of $T^*$ in this limit without proving global well-posedness for large data. In the work \cite{flynn2024massless}, the authors overcome this issue by relying on the proximity of $F_-$ to a quasi-static equilibrium solution. The fluid limit $\kappa \to 0$ for both \eqref{eq:VPLunit} and \eqref{eq:VPLion''} is similarly singular.

On that note, it is interesting to consider what the optimal lower bound on $T^*$ in terms of $\ep$ and $\kappa$. The methods of this paper could produce some (likely suboptimal) quantitative lower bound on $T^*$. On the other hand, with recent progress on the global well-posedness of the homogeneous Landau equation \cite{guillen2023landau}, it is plausible that the systems considered in this paper also have global well-posedness for large data.
\end{remark}

\begin{remark}
The number of moments in $v$ controlled by the $\mathfrak E$ and $\mathfrak D$ norms is far from optimal, and was chosen for the sake of simplicity.  On the other hand, with the method of proof in this paper,  the requirement that $s>\frac{5}{2}$ may be relaxed to $s > \frac{3}{2}$, potentially at the cost of higher $v$ moment control. Our method would  fail below this threshold, as we rely heavily on the Sobolev embedding $H^{s}(\mathbf T^3) \subset L^\infty(\mathbf T^3)$. If we consider $s \in (\frac{3}{2},\frac{5}{2})$, we would need control on additional $v$ moments to bound the commutator $[\langle \nabla_x\rangle^s,v\cdot \nabla_x]F$ via the same interpolation argument given in Step 3 of Proposition \ref{prop:aprioriFG}. Moreover, the amount of moments needed to close this argument diverges as $s \downarrow \frac{3}{2}$.

\end{remark}

\begin{remark}
Using this method, it is straightforward (in fact, easier) to prove a local well-posedness theorem for the inhomogeneous Landau equation,
\begin{align}
(\partial_t + v\cdot \nabla_x)F = Q(F,F),
\end{align}
in the same topology as in Theorem \ref{thm:LWP}. Our result is not a strict improvement on any previous result, due to our choice of domain. Nevertheless, in terms of regularity, the anisotropic $\mathfrak E$ norm used here is weaker than the norms used in \cite{he2014well} (which required 7 $(x,v)$ derivatives in $L^2$) and \cite{henderson2019local} (which required 4 derivatives). The work \cite{henderson2020local} is neither stronger nor weaker than our result, as they prove local well-posedness for data in a weighted $L^\infty$ space.

\end{remark}

\subsection{Method of the proof}

The proof method involves fairly straightforward propagation of the norm $\mathfrak E$ via energy estimates, and commutator estimates, as in Proposition \ref{prop:aprioriFG}. This relies on Lemma 3.2 from \cite{flynn2024massless}, restated as Lemma \ref{lem:PhiUpperLower} below, to control collisions. Then, we control differences of solutions (contraction estimates) as in Proposition \ref{prop:difference}. With these alone, we can construct solutions to the Vlasov-Poisson-Landau system \eqref{eq:VPLunit} as the limit of smooth approximate solutions given by an iteration procedure, given in Section \ref{sec:construction}.

To construct solutions to the system \eqref{eq:VPLion''}, the procedure is mostly the same, except for having to deal with the Poincar\'e-Poisson system that couples $(\beta,\phi)$ to $F_+$. In particular, we need to bound $\beta$ from above and below, and control $\phi \in H^{s+2}$, as well as control the various error terms associated with these unknowns in the iteration process. The culmination of this is Lemma \ref{lem:PPsystemEstimates}, which builds on previous analyses of this system in \cite{bardos_maxwellboltzmann_2018,flynn2024massless},

In attempting to reuse the estimates from \cite{flynn2024massless}, we encountered one main difficulty, foreshadowed in the remark above. In the previous paper, we are able to close estimates without any $v$ derivatives (and indeed, it was necessary). However, this framework was insufficient for general large data. This is because we were unable to close standard bootstrap assumptions without propagating $v$ derivatives, due to the commutators of $x$ derivatives with the diffusion matrix. With these estimates closed, we construct solutions by a standard iteration argument.

\section{\textit{A priori} estimates}
\label{sec:LWP}

\subsection{Lower and upper bounds on the diffusion matrix}
%
%

We briefly restate Lemma 3.2 from \cite{flynn2024massless} below without proof.

\begin{lemma} \label{lem:PhiUpperLower}
Let $G(v) = G :  \mathbf R^3 \to \mathbf R$ be some measurable function, and let $\nu \in \mathbf S^2$. Then for all $v \in \mathbf R^3$,
\begin{align}\label{eq:PhiConvUpper}
| \Phi_{ij} * G(v) \nu_i \nu_j| \lesssim \|\langle v\rangle^{5}G\|_{L^2} \sigma_{ij} (v)\nu_i \nu_j.
\end{align}
Assuming $G \geq 0$,  we have the lower  bound
\begin{align}\label{eq:PhiConvLower}
\frac{\|G\|_{L^1}}{\langle\|\langle v\rangle^{2}G\|_{L^2}/\|G\|_{L^1}\rangle ^{17} } \sigma_{ij}(v) \nu_i \nu_j \lesssim \Phi_{ij} * G(v) \nu_i \nu_j.
\end{align}

\end{lemma}

\subsection{Energy estimates}

As done in \cite{flynn2024massless}, we split the collision operator into ``diffusion" (second order) and ``transport" (first order) parts in divergence form:
\begin{align}
Q(G_1,G_2)&= Q_D(G_1,G_2) + Q_T(G_1,G_2)\\
Q_D(G_1,G_2)&:= \partial_{v_j} ((\Phi_{ij}* G_1) \partial_{v_i} G_2)\\
Q_T(G_1,G_2)&:=-  \partial_{v_j}(\partial_{v_i}(\Phi_{ij}* G_1) G_2)
\end{align}
Here, in the case where $G_1$ depends on $t$ or $x$, we abuse notation slightly and write 
\begin{align}
\Phi_{ij}*G_1(t,x,v) = \int_{\mathbf R^3} \Phi_{ij}(v-v')G_1(t,x,v')dv'. 
\end{align}

We now prove a priori estimates for the following system:
\begin{align}\label{eq:FG}
\partial_t F + \{v\cdot \nabla_x -\nabla_x \psi \cdot \nabla_v\} F = Q(G,F).
\end{align}

\begin{proposition}\label{prop:aprioriFG}
Let $M > 0$. 
Suppose $F$ is a weak solution to \highlightEdit{\eqref{eq:FG}} on the interval $t \in [0,T]$. Assume the following on $G = G(t,x,v) \geq 0$:
\begin{align}\label{eq:bootstrapG}
\sup_{t \in [0,T]} \left(\|\frac{1}{\int_{\mathbf R^3} G dv}\|_{L^\infty_x}+\|G\|_{\mathfrak E} + \| \psi\|_{H^{s+1}_x}\right) \leq M <\infty
\end{align}
Then,
\begin{align}\label{eq:apriori}
 \|F\|_{\mathfrak E_T\cap \mathfrak D_T} \leq e^{C_M (T + \sqrt{T}\|G\|_{\mathfrak D_T}^\frac{3}{2})} \|F_{in}\|_{\mathfrak E}.
\end{align}

%
\end{proposition}
\begin{proof} 
For every $m',s',r'$, we denote \highlightEdit{$F^{(m',s',r')} = \langle v\rangle^{m'} \langle \nabla_x\rangle^{s'} \langle \nabla_v\rangle^{r'} F$}, and similarly for $G^{(m',s',r')}$. 
\\

\noindent \textbf{Step 1:} We prove \highlightEdit{for all $\kappa >0$,}
\begin{align}\label{eq:vWeightEnergy}
\frac{d}{dt}\| F^{(m_2,0,0)}\|_{\highlightEdit{L^2_{x,v}}}^2 + \frac{1}{C} \|F^{(m_2,0,0)}\|_{\highlightEdit{L^2_{x}(\dot{\mathcal H}_\sigma)_v}}^2 \leq   C_\kappa \langle \|G\|_{\mathfrak D}\rangle^{\frac{3}{2}} \|F\|_{\mathfrak E}^2 + \kappa \|F\|_{\mathfrak D}^2.
\end{align}
Observe that $F^{(m_2,0,0)}$ this satisfies 
\begin{align}
&(\partial_t  + v\cdot \nabla_x - \nabla_x \psi  \cdot \nabla_v )F^{(m_2,0,0)}- [\langle v\rangle^{m_2}, \nabla_x \psi \cdot \nabla_v] F =\langle v\rangle^{m_2} Q(G,F).
\end{align}
Thus,
\begin{align}
\frac{1}{2} \frac{d}{dt}\|F^{(m_2,0,0)}\|_{\highlightEdit{L^2_{x,v}}}^2 & =\langle [\langle v\rangle^{m_2}, \nabla_x \psi \cdot \nabla_v] F,F^{(m_2,0,0)}\rangle_{L^2_{x,v}} \label{eq:apriori_m2T1} \\
&\quad +\langle \langle v\rangle^{m_2}Q_D(G,F),F^{(m_2,0,0)}\rangle_{L^2_{x,v}}\label{eq:apriori_m2T2}\\
&\quad +\langle \langle v\rangle^{m_2} Q_T(G,F),F^{(m_2,0,0)}\rangle_{L^2_{x,v}}.\label{eq:apriori_m2T3}
\end{align}
First, we see that
\begin{align}
\eqref{eq:apriori_m2T1} & =-m_2\langle \nabla_x \psi\frac{v}{\langle v\rangle^2}F^{(m_2,0,0)},F^{(m_2,0,0)}\rangle_{L^2_{x,v}}\\
& \lesssim \|\nabla_x \psi\|_{L^\infty_x} \|F^{(m_2,0,0)}\|_{L^2_{x,v}}^2\\
&\lesssim \|\psi\|_{H^s_x} \|F^{(m_2,0,0)}\|_{L^2_{x,v}}^2 \\
&\lesssim  \|F^{(m_2,0,0)}\|_{L^2_{x,v}}^2 .
\end{align}
Next,
\begin{align}
\langle v\rangle^{m_2} \partial_{v_j}( \Phi_{ij} * G \partial_{v_i} F)
&= \partial_{v_j}( \Phi_{ij} * G \partial_{v_i} F^{(m_2,0,0)}\\
&\quad  - m_2 \frac{v_j}{\langle v\rangle^2} \Phi_{ij} * G \partial_{v_i} F^{(m_2,0,0)}\\
&\quad   - m_2\partial_{v_j}( \frac{v_i}{\langle v\rangle^2} \Phi_{ij} * GF^{(m_2,0,0)}) \\
&\quad + m_2(m_2-2)\frac{v_iv_j}{\langle v\rangle^{4}} \Phi_{ij} * G  F.
\end{align}
Thus,
\begin{align}
\eqref{eq:apriori_m2T2} &\leq - \langle \Phi_{ij}* G \partial_{v_i} F^{(m_2,0,0)},\partial_{v_j} F^{(m_2,0,0)}\rangle_{L^2_{x,v}} \\
&\quad +C\|\langle v\rangle^{\frac{1}{2}}\Phi_{ij}*G\|_{L^\infty_{x,v}} \|F^{(m_2,0,0)}\|_{L^2_{x,v}}\|F^{(m_2,0,0)}\|_{L^2_x(\dot{\mathcal H}_\sigma)_v}\\
& \quad   +C\|\Phi_{ij}*G\|_{L^\infty_{x,v}} \|F^{(m_2,0,0)}\|_{L^2_{x,v}}^2
\end{align}
Using Lemma \ref{lem:PhiUpperLower},
\begin{align}
 \langle \Phi_{ij}* G \partial_{v_i} F^{(m_2,0,0)},\partial_{v_j} F^{(m_2,0,0)}\rangle_{L^2_{x,v}} \gtrsim  \|F^{(m_2,0,0)}\|_{L^2_x(\dot{\mathcal H}_\sigma)_v}^2, 
 \end{align}
 and
 \begin{align}
 \|\langle v\rangle^{\frac{1}{2}}\Phi_{ij}*G\|_{L^\infty_{x,v}} \lesssim \|\langle v\rangle^5 G\|_{L^\infty_x L^2_v} \lesssim \|G^{(m_1,s,0)}\|_{L^2_{x,v}}.
 \end{align}
 Hence,
\begin{align}
\eqref{eq:apriori_m2T2} &\leq - \frac{1}{C} \|F^{(m_2,0,0)}\|_{L^2_x(\dot{\mathcal H}_\sigma)_v}^2+ C \|F^{(m_2,0,0)}\|_{L^2_{x,v}}^2.
\end{align}
Next, we have
\begin{align}
\eqref{eq:apriori_m2T3}& = \langle \partial_{v_i} \Phi_{ij} * G F^{(m_2,0,0)},  \partial_{v_j} F^{(m_2,0,0)}\rangle_{L^2_{x,v}} \\
&\quad +  m\langle \frac{v}{\langle v\rangle^2} \partial_{v_i} \Phi_{ij} * G F^{(m_2,0,0)},  F^{(m_2,0,0)}\rangle_{L^2_{x,v}} \\
&=\|\langle v\rangle^\frac{3}{2} \nabla_v \Phi* G\|_{L^\infty_{x,v}}\| F^{(m_2,0,0)}\|_{L^2_{x}(L^2\cap \dot{\mathcal H}_\sigma)_v}  \| F^{(m_2,0,0)}\|_{L^2_{x,v}}.
\end{align}
Now, we claim that
\begin{align}\label{eq:gradPhiBd}
\|\langle v\rangle^2\nabla_v \Phi* G\|_{L^\infty_{x,v}} \lesssim \|G\|_{\mathfrak E} + \|G\|_{\mathfrak E}^{\frac{1}{4}} \|G\|_{\mathfrak D}^{\frac{3}{4}}.
\end{align}
To see this, first observe that $|\nabla_v \Phi| \lesssim \frac{1}{|v|^2}$. Thus, we have for all $(t,x,v)$,
\begin{align}
|\langle v\rangle^2\partial_{v_i} \Phi_{ij}* G (t,x,v)|& \leq \langle v\rangle^2 \int_{|v| > 2 \langle v'\rangle} \frac{1}{|v-v'|^2} |G(t,x,v')|dv' \\
&\quad  +\langle v\rangle^2  \int_{|v| \leq 2 \langle v'\rangle} \frac{1}{|v-v'|^2} |G(t,x,v')|dv' \\
&\lesssim \int_{\mathbf R^3} G(t,x,v') dv' + |\nabla_v |^{-1} \{\langle v\rangle^2 | G|\}(t,x,v)
\end{align}
Hence,
\begin{align}
\|\langle v\rangle^2\partial_{v_i} \Phi_{ij}* G\|_{L^\infty_{x,v}} &\lesssim \|G\|_{L^\infty_xL^1_v} + \||\nabla_v |^{-1} \{\langle v\rangle^2 | G|\}\|_{L^\infty_{x,v}} \\
&\lesssim \|G^{(2,s,0)}\|_{L^2_{x,v}} +\| \ |G^{(2,0,0)}|\ \|_{L^\infty_x H^{3/4}_v}\\
&\lesssim \|G^{(m_1,s,0)}\|_{L^2_{x,v}} + \|G^{(2,0,0)}\|_{L^\infty_x L^2_{v}}^{\frac{1}{4}} \|G^{(2,0,0)}\|_{L^\infty_x H^1_v}^{\frac{3}{4}}\\
&\lesssim  \|G\|_{\mathfrak E}  +  \|G\|_{\mathfrak E}^{\frac{1}{4}} \|G\|_{\mathfrak D}^{\frac{3}{4}}.
\end{align}
In the above, we used the fact that $\|  \  |u| \ \|_{H^1} =\|u\|_{H^1}$. Combining these estimates on \eqref{eq:apriori_m2T1},  \eqref{eq:apriori_m2T2}, and  \eqref{eq:apriori_m2T3}, we get \eqref{eq:vWeightEnergy}.\\

\noindent \textbf{Step 2:}
We now prove  \highlightEdit{for all $\kappa >0$,}
\begin{align}\label{eq:vRegEnergy}
\frac{d}{dt}\|F^{(0,0,r)}\|_{L^2{x,v}}^2 + \frac{1}{C} \|F^{(0,0,r)}\|_{L^2{x}(\mathcal H_\sigma)_v}^2 \leq   C_\kappa \langle \|G\|_{\mathfrak D}\rangle^{\frac{3}{2}} \|F\|_{\mathfrak E}^2 + \kappa \|F\|_{\mathfrak D}^2.
\end{align}
First, we compute
\begin{align}
\frac{1}{2} \frac{d}{dt}\|F^{(0,0,r)}\|_{L^2{x,v}}^2 & =\langle [\langle\nabla_v\rangle^r,v\cdot \nabla_x] F,F^{(0,0,r)}\rangle_{L^2_{x,v}} \label{eq:apriori_rT1} \\
&\quad +\langle \langle \nabla_v\rangle^{r}Q_D(G,F),F^{(0,0,r)}\rangle_{L^2_{x,v}}\label{eq:apriori_rT2}\\
&\quad +\langle \langle \nabla_v\rangle^{r}Q_T(G,F),F^{(0,0,r)}\rangle_{L^2_{x,v}}.\label{eq:apriori_rT3}.
\end{align}
We now bound \eqref{eq:apriori_rT1}.
\begin{align}
[\langle\nabla_v\rangle^r,v\cdot \nabla_x] F =- r\langle \nabla_v\rangle^{r-2}\nabla_v\cdot \nabla_x F.
\end{align}
Hence,
\begin{align}
\eqref{eq:apriori_rT1} \lesssim \|\langle \nabla_v\rangle^{r-1} \nabla_x F \|_{L^2_{x,v}} \|F^{(0,0,r)}\|_{L^2_{x,v}}\lesssim  \|F\|_{\mathfrak E}^2.
\end{align}
To bound the remaining terms, we shall prove the following weighted commutator estimate:
Given appropriate $u_1(v), u_2(v)$, and $m \geq 0$, we have
\begin{align}\label{eq:weightedComm}
\|\langle v\rangle^m [\langle \nabla_v\rangle^r,u_1]u_2\|_{L^2_v} &\lesssim \|\langle \nabla_v\rangle^r u_1\|_{ (\mathcal F {L}^1)_v} \|{\langle v\rangle^m} u_2\|_{L^2_v} \\
&\lesssim \|\langle \nabla_v\rangle^{r+\frac{3}{2}} u_1\|_{L^2_v} \|\highlightEdit{\langle v\rangle^m} u_2\|_{L^2_v},
\end{align}
where $\mathcal F L^1$ is the space of functions whose Fourier transform is  in $L^1$. As usual, we define the Fourier transform as
\begin{align}
\hat u(\xi) = \mathcal F\{u\}(\xi)= \int_{\mathbf R^d} e^{-i v\cdot \xi} u(v)dv.
\end{align}
 By interpolation, it suffices to show \eqref{eq:weightedComm} when $m \in 2\mathbf N_0$. We take the Fourier transform,
\begin{align}
&\mathcal F\{ \langle v\rangle^m [\langle \nabla_v\rangle^r,u_1]u_2\}(\xi)\\
 &= (1- \Delta_\xi)^{\frac{m}{2}} \int (\langle \xi \rangle^r - \langle \xi - \eta\rangle^r) \hat u_1 (\eta) \hat u_2(\xi - \eta) d\eta \\
&=   \int (\langle \xi \rangle^r - \langle \xi - \eta\rangle^r) \hat u_1 (\eta) (1- \Delta_\xi)^{\frac{m}{2}}\hat u_2(\xi - \eta) d\eta \\
&\quad + \sum_{|\alpha + \alpha' | =\frac{m}{2}, \alpha \neq 0} c_{\alpha,\alpha'} \int \partial_\xi^\alpha (\langle \xi \rangle^r - \langle \xi - \eta\rangle^r) \hat u_1 (\eta) \partial_\xi^{\alpha'} \hat u_2(\xi - \eta)d\eta 
\end{align}
where $c_{\alpha,\alpha'}$ are some collection of constants.
Next, $|\langle \xi\rangle - \langle \xi - \eta\rangle^r | \leq \langle \eta \rangle^r$ (recall $r \in (0,1]$). Moreover, $|\partial^\alpha_\xi  (\langle \xi \rangle^r - \langle \xi - \eta\rangle^r)| \lesssim 1$ for all multi-indices $\alpha \neq 0$. Thus, pointwise in $\xi$, we have
\begin{align}
|\mathcal F\{ \langle v\rangle^m [\langle \nabla_v\rangle^r,u_1]u_2\}(\xi)| &\lesssim  \langle \eta \rangle^r|\hat u_1 (\eta) (1- \Delta_\xi)^{\frac{m}{2}}\hat u_2(\xi - \eta) |d\eta  \\
&\quad  +\sum_{|\alpha + \alpha' | =\frac{m}{2}, \alpha \neq 0}  \int | \hat u_1 (\eta) \partial_\xi^{\alpha'} \hat u_2(\xi - \eta)|d\eta 
\end{align}
The estimate \eqref{eq:weightedComm} directly follows from Young's inequality, and Plancherel's identity.

We now apply this estimate. First, split
\begin{align}
\langle \nabla_v\rangle^{r}Q_D(G,F) =  Q_D(G,F^{(0,0,r)}) + \nabla_v \cdot ([\langle \nabla_v\rangle^r, \Phi*G] \nabla_v F).
\end{align}
Therefore,
\begin{align}
\eqref{eq:apriori_rT2} \leq - \frac{1}{C} \|F^{(0,0,r)}\|_{L^2_x(\mathcal H_\sigma)_v}^2  + \langle [\langle \nabla_v\rangle^r, \Phi*G] \nabla_v F, \nabla_v F^{(0,0,r)}\rangle_{L^2_{x,v}}.
\end{align}
The latter term is bounded as follows,
\begin{align}
& \langle [\langle \nabla_v\rangle^r, \Phi*G] \nabla_v F, \nabla_v F^{(0,0,r)}\rangle_{L^2_{x,v}} & \\
 &\quad \lesssim \|\langle v\rangle^{\frac{3}{2}}[\langle \nabla_v\rangle^r, \Phi*G] \nabla_v F\|_{L^2_{x,v}} \|\langle v\rangle^{-\frac{3}{2}} \nabla_v F^{(0,0,r)}\|_{L^2_{x,v}} \\
 &\quad \lesssim  \|\Phi*G\|_{L^2_x\mathcal FL^1} \|\langle v\rangle^{\frac{3}{2}}\nabla_v F\|_{L^2_{x,v}} (\| F^{(0,0,r)}\|_{L^2_{x}(\dot{\mathcal H}_\sigma)_v} + \|F\|_{\mathfrak E}),
\end{align}
having applied \eqref{eq:weightedComm} in the last line.
Now, using $\hat G(t,x,\xi)$ to denote the Fourier transform of $G$ in $v$, we have
\begin{align}
 \|\Phi*G\|_{L^2_x(\mathcal FL^1)_v}& \lesssim  \|\int\frac{1}{ |\xi|^2} |\hat G(t,x,\xi) | d\xi\|_{L^2_x} \\
 &\leq  \|\int_{B(0,1)} \frac{1}{ |\xi|^2} |\hat G(t,x,\xi) | d\xi\|_{L^2_x} +  \|\int_{B(0,1)^c} \frac{1}{ |\xi|^2} |\hat G(t,x,\xi) | d\xi\|_{L^2_x} \\
 &\leq \|\hat G\|_{L^2_x L^\infty_\xi} + \|\hat G\|_{L^2_{x,v}} \|\frac{1_{B(0,1)^c}}{|\xi|^2}\|_{L^2_{\xi}}  \\
 &\lesssim \|\langle v\rangle^{4} G\|_{L^2_{x,v}} \\
 &\lesssim 1
\end{align}
In the penultimate line, we used Sobolev embedding \highlightEdit{$ H^4\subset L^\infty $}.  Combining all the bounds, we get
\begin{align}
\eqref{eq:apriori_rT2} \lesssim - \frac{1}{C} \|F^{(0,0,r)}\|_{L^2_x(\mathcal H_\sigma)_v}^2 +  \|\langle v\rangle^{\frac{3}{2}}\nabla_v F\|_{L^2_{x,v}} (\| F^{(0,0,r)}\|_{L^2_{x}(\mathcal H_\sigma)_v} + \|F\|_{\mathfrak E}).
\end{align}
Next, we verify the following claim: for all $\kappa >0$, 
\begin{align}
\|\langle v\rangle^{\frac{3}{2}}\nabla_v F^{(m_0,0,0)}\|_{L^2_{x,v}}  \leq C_\kappa \|F\|_{\mathfrak E} + \kappa \|F\|_{\mathfrak D}. 
\end{align}
First,  we apply the interpolation inequality: since $m_2 -\frac{3}{2}> \frac{3}{2}$, 
\begin{align}\label{eq:vWeightInterpolation}
  \|\langle v\rangle^{\frac{3}{2}}\nabla_v F^{(m_0,0,0)}\|_{L^2_{x,v}} \lesssim \| \langle v\rangle^{m_2-\frac{3}{2}}\nabla_v F^{(m_0,0,0)} \|_{L^2_{x,v}}^{1-\theta} \|\langle v\rangle^{-\frac{3}{2}} \nabla_v F^{(m_0,0,0)}\|_{L^2_{x,v}}^{\theta}
\end{align}
for some $\theta \in (0,1)$. The first factor is controlled by $\|F^{(m_2,0,0)}\|_{L^2_x(L^2\cap \dot{\mathcal H}_\sigma)_v} \lesssim \|F\|_{\mathfrak E\cap \mathfrak D}$. 

We further interpolate the second factor as follows:
\begin{align}\label{eq:grad_vInterpolation}
\|\langle v\rangle^{-\frac{3}{2}} \nabla_v F\|_{L^2_{x,v}} &\lesssim \|\langle \nabla_v \rangle^{-1}(\langle v\rangle^{-\frac{3}{2}} \nabla_v F)\|_{L^2_{x,v}}^{\frac{r}{1+r} }\\
&\quad \cdot \|\langle \nabla_v \rangle^{r}(\langle v\rangle^{-\frac{3}{2}} \nabla_v F^{(m_0,0,0)})\|_{L^2_{x,v}}^{\frac{1}{1+r}}
\end{align}
Writing $\langle v\rangle^{-\frac{3}{2}} \nabla_v F  = \nabla_v(\langle v\rangle^{-\frac{3}{2}} F) - \nabla_v\langle v\rangle^{-\frac{3}{2}} F$, we see that the first factor in the above is bounded by $\|F\|_{L^2_{x,v}}$. 

%

As for the second, we use \eqref{eq:weightedComm},
\begin{align}
\|\langle \nabla_v \rangle^{r}(\langle v\rangle^{-\frac{3}{2}} \nabla_v F)\|_{L^2_{x,v}} &\lesssim \|\langle v\rangle^{-\frac{3}{2}} \nabla_v \langle \nabla_v\rangle^r F\|_{L^2_{x,v}}\\
&\quad  + \|[\langle \nabla_v \rangle^{r},\langle v\rangle^{-\frac{3}{2}}] \nabla_v F\|_{L^2_{x,v} }\\
&\lesssim \|F^{(0,0,r)}\|_{L^2_x(\dot{\mathcal H}_\sigma)_v} + \|\langle \nabla_v\rangle^{r} \langle v\rangle^{-\frac{3}{2}}\|_{\mathcal FL_1}\|\nabla_v F\|_{L^2_{x,v}}
\end{align}
Now, $\mathcal F\{\langle v\rangle^{-\frac{3}{2}}\}(\xi)$ behaves like $|\xi|^{-\frac{3}{2}}$ near the origin, and decays rapidly as $|\xi|\to \infty$. Thus, 
\begin{align}
\|\langle \nabla_v \rangle^{r}(\langle v\rangle^{-\frac{3}{2}}\nabla_v F)\|_{L^2_{x,v}} \lesssim  \|F^{(0,0,r)}\|_{L^2_x(\dot{\mathcal H}_\sigma)_v} +\|\nabla_v F\|_{L^2_{x,v}}.
\end{align}
Combining with \eqref{eq:vWeightInterpolation} and \eqref{eq:grad_vInterpolation}, we get
\begin{align}
   \|\langle v\rangle^{\frac{3}{2}}\nabla_v F\|_{L^2_{x,v}}  \lesssim \| \langle v\rangle^{m_2-\frac{3}{2}}\nabla_v F \|_{L^2_{x,v}}^{1-\theta} \|F\|_{L^2_{x,v}}^{\frac{\theta r}{1+r}} (\|F^{(0,0,r)}\|_{L^2_x(\dot{\mathcal H}_\sigma)_v} +\|\nabla_v F\|_{L^2_{x,v}})^{\frac{\theta}{1+r}}.
\end{align}
Hence, for all $\kappa$, we can bound
\begin{align}\label{eq:gradv_RegInterp}
   \|\langle v\rangle^{\frac{3}{2}}\nabla_v F\|_{L^2_{x,v}}  \leq C_\kappa \|F\|_{\mathfrak E} + \kappa \|F\|_{\mathfrak D}.
\end{align}
We conclude that
\begin{align}
\eqref{eq:apriori_rT2} \leq - \frac{1}{C} \|F^{(0,0,r)}\|_{L^2_x(\mathcal H_\sigma)_v}^2 +  C_\kappa \|F\|_{\mathfrak E}^2 + \kappa \|F\|_{\mathfrak D}^2.
\end{align}
Finally, we bound \eqref{eq:apriori_rT3}. We write
\begin{align}
\eqref{eq:apriori_rT3}  = \langle \partial_{v_i} \Phi_{ij} * G F^{(0,0,r)},\partial_{v_j} F^{(0,0,r)}\rangle_{L^2_{x,v}} + \langle    [\langle \nabla_v\rangle^r,\partial_{v_i}\Phi_{ij} * G] F,\partial_{v_j} F^{(0,0,r)}\rangle_{L^2_{x,v}} 
\end{align}
Similarly to bounding \eqref{eq:apriori_m2T2}, 
\begin{align}
 \langle \partial_{v_i} \Phi_{ij} * G F^{(0,0,r)},\partial_{v_j} F^{(0,0,r)}\rangle_{L^2_{x,v}}\lesssim \langle \|G\|_{ \mathfrak D}\rangle^\frac{3}{4} \|F\|_{\mathfrak E}\|F\|_{\mathfrak E\cap \mathfrak D}.
\end{align}
On the other hand, using \eqref{eq:weightedComm}, we have
\begin{align}
\langle  [\langle \nabla_v\rangle^r,\partial_{v_i}\Phi_{ij} * G] F,\partial_{v_j} F^{(0,0,r)}\rangle_{L^2_{x,v}}& \lesssim \|\langle v\rangle^2 [\langle \nabla_v\rangle^r,\partial_{v_i}\Phi_{ij} * G] F\|_{L^2_{x,v}} \|F\|_{\mathfrak E\cap \mathfrak D} \\
&\lesssim \||\nabla_v|^{-1} \langle \nabla_v\rangle^r G\|_{L^\infty_x(\mathcal FL^1)_v} \|F\|_{\mathfrak E} \|F\|_{\mathfrak E\cap \mathfrak D}.
\end{align}
Similarly to the bound on $\Phi*G \in \mathcal FL^1$, we have
\begin{align}
\||\nabla_v|^{-1} \langle \nabla_v\rangle^r G\|_{L^\infty_x(\mathcal FL^1)_v} \lesssim  \langle \|G\|_{ \mathfrak D}\rangle^\frac{3}{4} .
\end{align}
Thus, altogether,
\begin{align}
\eqref{eq:apriori_rT3} \lesssim C_\kappa \langle \|G\|\rangle^{\frac{3}{2}}\|F\|_{\mathfrak E}^2 + \kappa \|F\|_{\mathfrak D}^2.
\end{align}
 \\

\noindent \textbf{Step 3:} We now prove that for all $\kappa >0$,
\begin{align}\label{eq:xRegEnergy}
\frac{d}{dt} \|F^{(m_1,s,0)}\|_{L^2_{x,v}}^2  + \frac{1}{C} \|F^{(m_1,s,0)}\|_{L^2_{x}(\dot{\mathcal H}_\sigma)_v}^2 \leq C_\kappa \langle \|G\|\rangle^{\frac{3}{2}}\|F\|_{\mathfrak E}^2 + \kappa \|F\|_{\mathfrak D}^2.
\end{align}
Observe that
 $F^{(m_1,s,0)}$ satisfies 
\begin{align}
&(\partial_t  + v\cdot \nabla_x - \nabla_x \psi  \cdot \nabla_v )F^{(m_1,s,0)} - [\langle v\rangle^{m_1} \langle \nabla_x\rangle^s, \nabla_x \psi \cdot \nabla_v] F \\
& \quad =\langle v\rangle^{m_1} \langle \nabla_x\rangle^s Q(G,F).
\end{align}
Then,
\begin{align}
\frac{1}{2} \frac{d}{dt}\|F^{(m_1,s,0)}\|_{L^2{x,v}}^2 & =\langle [\langle v\rangle^{m_1} \langle \nabla_x\rangle^s, \nabla_x \psi \cdot \nabla_v] F,F^{(m_1,s,0)}\rangle_{L^2_{x,v}} \label{eq:aprioriT1} \\
&\quad +\langle \langle v\rangle^{m_1} \langle \nabla_x\rangle^s Q_D(G,F),F^{(m_1,s,0)}\rangle_{L^2_{x,v}}\label{eq:aprioriT2}\\
&\quad +\langle \langle v\rangle^{m_1} \langle \nabla_x\rangle^s Q_T(G,F),F^{(m_1,s,0)}\rangle_{L^2_{x,v}}.\label{eq:aprioriT3}
\end{align}
Now,
\begin{align}
\eqref{eq:aprioriT3} &=  \langle \langle v\rangle^{m_1} [ \langle \nabla_x\rangle^s, \nabla_x \psi]\cdot \nabla_v F,F^{(m_1,r)}\rangle_{L^2_{x,v}}\\
&\quad + \langle [\langle v\rangle^{m_1} , \nabla_x \psi \cdot \nabla_v]F^{(0,s)},F^{(m_1,r)}\rangle_{L^2_{x,v}}\\
&\lesssim \|\nabla_x\psi\|_{H^{s}_x}(\|\langle v\rangle^{m_1} \nabla_v F^{(0,s-1)}\|_{L^2_{x,v}} +   \|F^{(m_1,s,0)}\|_{L^2_{x,v}}) \|F^{(m_1,s,0)}\|_{L^2_{x,v}}.
\end{align}
Now,
\begin{align}
\|\langle v\rangle^{m_1} \nabla_v F^{(0,s-1)}\|_{L^2_{x,v}} & \lesssim \|\langle v\rangle^{m_1 + \frac{3}{2}(s-1)} \nabla_v F\|_{L^2_{x,v}}^{\frac{1}{s}} \|\langle v\rangle^{m_1-\frac{3}{2}} \nabla_v F^{(0,s,0)}\|_{L^2_{x,v}}^{\frac{s-1}{s}}  \\
&\lesssim   \| F^{(m_1,s,0)}\|_{L^2_x(L^2 \cap \dot{\mathcal H}_\sigma)_v}.
\end{align}
Thus,
\begin{align}
\eqref{eq:aprioriT1} \lesssim \|F\|_{\mathfrak E\cap \mathfrak D} \|F\|_{\mathfrak E}.
\end{align}
Next, we have
\begin{align}
\eqref{eq:aprioriT2} &=- \langle \Phi_{ij} *G \partial_{v_i} F^{(m_1,s,0)}, \partial_{v_j} F^{(m_1,s,0)}\rangle_{L^2_{x,v}}\label{eq:aprioriT2T1} \\
&\quad - \langle \langle \nabla_x\rangle^s\{ \Phi_{ij} *G \partial_{v_i} F^{(m_1,0,0)}\} - \Phi_{ij} *G \partial_{v_i} F^{(m_1,s,0)}, \partial_{v_j} F^{(m_1,s,0)}\rangle_{L^2_{x,v}} \label{eq:aprioriT2T2}\\
&\quad +m_1 \langle\langle \nabla_x\rangle^s \{ \Phi_{ij} *G \frac{v_i}{\langle v\rangle^2} F^{(m_1,0,0)}\}, \partial_{v_j}F^{(m_1,s,0)} \rangle_{L^2_{x,v}} \label{eq:aprioriT2T3}\\
&\quad -m_1 \langle v_j\langle v\rangle^{m_1-2}\langle \nabla_x\rangle^s\{ \Phi_{ij} *G \partial_{v_i}F\}, F^{(m_1,s,0)}\rangle_{L^2_{x,v}} \label{eq:aprioriT2T4}.
\end{align}
For \eqref{eq:aprioriT2T1}, we use \eqref{eq:PhiConvLower} to get
\begin{align}
\eqref{eq:aprioriT2T1} \leq -\frac{1}{C} \|F^{(m_1,s,0)}\|_{L^2_x(\dot{\mathcal H}_\sigma)_v}^2.
\end{align}
Using \eqref{eq:PhiConvUpper} we have
\begin{align}
 \eqref{eq:aprioriT2T2} &\lesssim \|\sigma^{-\frac{1}{2}} \Phi* G^{(0,s)}\sigma^{-\frac{1}{2}}\|_{L^\infty_{x,v}} \|F^{(m_1,s-1,0)}\|_{L^2_x(\dot{\mathcal H}_\sigma)_v} \|F^{(m_1,s,0)}\|_{L^2_x(\dot{\mathcal H}_\sigma)_v}  \\
 &\lesssim \|G\|_{\mathfrak E} \|F^{(m_1,0,0)}\|_{L^2_x(\dot{\mathcal H}_\sigma)_v}^{\frac{1}{s}} \|F^{(m_1,s,0)}\|_{L^2_x(\dot{\mathcal H}_\sigma)_v} ^{2 - \frac{1}{s}}.
\end{align}
Similarly as with \eqref{eq:gradv_RegInterp}, we can bound 
\begin{align}
\|F^{(m_1,0,0)}\|_{L^2_x(\dot{\mathcal H}_\sigma)_v} &\lesssim \|F^{(m_1,0,0)}\|_{L^2_{x,v}} + \|\langle v\rangle^{m_1 - \frac{1}{2}} \nabla_v F\|_{L^2_{x,v}} \\
&\leq 
C_\kappa \|F\|_{\mathfrak E} + \kappa \|F\|_{\mathfrak D}.
\end{align}
for all $\kappa >0$.
Thus,
\begin{align}
 \eqref{eq:aprioriT2T2} &\leq C_\kappa \|F\|_{\mathfrak E}^2 + \kappa \|F\|_{\mathfrak D}^2.
\end{align}
Finally, we  have
\begin{align}
\eqref{eq:aprioriT2T3}  + \eqref{eq:aprioriT2T4}& \leq C \|G\|_{\mathfrak E} \|F^{(m_1,s,0)}\|_{L^2_x(\dot{\mathcal H}_\sigma)_v}  \|F^{(m_1,s,0)}\|_{L^2_{x,v}}\\
&\leq C_\kappa \|F\|_{\mathfrak E}^2 + \kappa \|F\|_{\mathfrak D}^2.
\end{align}
Thus, altogether,
\begin{align}
\eqref{eq:aprioriT2} \leq -\frac{1}{C} \|F^{(m_1,s,0)}\|_{L^2_x(\dot{\mathcal H}_\sigma)_v}^2 + C_\kappa \|F\|_{\mathfrak E}^2 + \kappa \|F\|_{\mathfrak D}^2.
\end{align}
Finally,
\begin{align}
\eqref{eq:aprioriT3} &= \langle \langle \nabla_x\rangle^s \{\partial_{v_i} \Phi_{ij} * GF^{(m_1,0,0)}\}, \partial_{v_j}F^{(m_1,s,0)}  \rangle_{L^2_{x,v}}\\
&\quad  + m_1  \langle \frac{v_j}{\langle v\rangle^2} \langle \nabla_x\rangle^s \{\partial_{v_i} \Phi_{ij} * GF^{(m_1,0,0)}\}, F^{(m_1,s,0)}  \rangle_{L^2_{x,v}} \\
& \lesssim \|\langle v\rangle^2\nabla_v \Phi *G^{(0,s)}\|_{L^2_{x}L^\infty_v}\|F\|_{\mathfrak E \cap \mathfrak D}\|F\|_{\mathfrak E} \\
&\lesssim \|G\|_{\mathfrak E}^{\frac{1}{4}} \|G\|_{\mathfrak E \cap \mathfrak D}^{\frac{3}{4}}\|F\|_{\mathfrak E \cap \mathfrak D}\|F\|_{\mathfrak E}\\
&\lesssim  C_\kappa \langle \|G\|_{\mathfrak E \cap \mathfrak D}\rangle^{\frac{3}{2}}\|F\|_{\mathfrak E}^2 + \kappa \|F\|_{\mathfrak D}^2.
\end{align}
In the above, we used a similar estimate as in \eqref{eq:gradPhiBd}.
Combining the bounds for \eqref{eq:aprioriT1}, \eqref{eq:aprioriT2} and \eqref{eq:aprioriT3}, we have proved \eqref{eq:xRegEnergy}. Combining with the previous steps, we conclude with \eqref{eq:apriori}.\\

\noindent \textbf{Step 4:} Adding \eqref{eq:vWeightEnergy}, \eqref{eq:vRegEnergy} and \eqref{eq:xRegEnergy}, and taking $\kappa$ sufficiently small, we have
\begin{align}
\frac{d}{dt} \|F\|_{\mathfrak E}^2 + \frac{1}{C} \|F\|_{\mathfrak D}^2 \lesssim \langle \|G\|_{\mathfrak D}\rangle^{\frac{3}{2}} \|F\|_{\mathfrak E}^2.
\end{align}
Then, \eqref{eq:apriori} follows from Gr\"onwall's inequality.

\end{proof}

\subsection{Contraction estimates}
In this section, we prove analogous bounds from the previous section on differences of solutions to  \eqref{eq:FG}. These bounds shall be used to show that the approximate solutions constructed in Section \ref{sec:construction} converge to a true solution.

\begin{proposition}\label{prop:difference}

Let $M >0$, and  $G_\alpha(t,x,v) \geq 0$ and $\psi_\alpha(t,x)$ with $\alpha \in \{1,2\}$ satisfy
\begin{align}\label{eq:bootstrap_alpha}
\sup_{t \in [0,T], \alpha \in\{1,2\}} \left(\|\frac{1}{\int_{\mathbf R^3} G_\alpha dv}\|_{L^\infty_x}+\|G_\alpha\|_{\mathfrak E} + \| \psi_\alpha\|_{H^{s+1}_x}\right) \leq M\highlightEdit{.}
\end{align}
Suppose $F_\alpha$ is a solution to \eqref{eq:FG} with $\psi = \psi_\alpha$ and $G = G_\alpha$ and initial data $F_{\alpha,in}$. Let $ \delta F = F_2 - F_1$, $ \delta G = G_2 - G_1$ and $\delta \psi = \psi_2 - \psi_1$. Then,
\begin{equation}\label{eq:contraction}
\begin{aligned}
&\|\delta F\|_{\mathfrak E'_T}^2+  \|\delta F\|_{\mathfrak D'_T}^2\\
  &\quad \lesssim_M e^{C_M (T + \sqrt{T} \|G_2\|_{\mathfrak D_T}^\frac{3}{2} )} \{ \|\delta F_{in}\|_{\mathfrak E'}^2  \\
&\quad \quad + \int_0^T (\|\delta\psi(t)\|_{H^{s+1}}^2+ \|\delta G(t)\|_{\mathfrak E'}^2)\|F_1(t)\|_{\mathfrak E \cap \mathfrak D'}^2  dt\\
&\quad \quad +  \int_0^T \|\delta G(t)\|_{\mathfrak E'}^\frac{1}{2} \|\delta G(t)\|_{ \mathfrak D'}^\frac{3}{2}\|F_1(t)\|_{ \mathfrak E} ^2 dt\}.
\end{aligned}
\end{equation}

\end{proposition}
\begin{proof} 
Similarly as in Proposition \ref{prop:aprioriFG}, we denote  $F^{(m',r',s')} = \langle v\rangle^{m'} \langle \nabla_x\rangle^{s'} \langle \nabla_v\rangle^{r'} F$, and similarly for $G$, $\delta G$ and $\delta F$. 
  Then, $\delta F$ solves 
\begin{align}
&\partial_t \delta F + \{v\cdot \nabla_x - \nabla_x \psi_2 \cdot \nabla_v\}\delta F - Q(G_2, \delta F) = Q(\delta G,F_1) + \nabla_x \delta  \psi \cdot \nabla_v F_1.
\end{align}
Observe that $\delta F$ satisfies the same type of equation as \eqref{eq:FG}, but with the forcing term $Q(\delta G, F_1) + \nabla_x \delta \psi\cdot \nabla_v F_1$. We note here that the proof of Proposition \ref{prop:aprioriFG} requires that $5\leq m_1$ and $m_1 + \frac{3}{2} s\leq m_2$. Thus, 
modifying the proof of \eqref{eq:vWeightEnergy}, we have
\begin{align}
&\frac{d}{dt} \|\delta F^{(m_1,0,0)}\|_{L^2_{x,v}}^2 + \frac{1}{C} \|\delta F^{(m_1,0,0)}\|_{L^2_{x}(\dot{\mathcal H}_\sigma)_v}^2\\
&\quad \lesssim (1 + \|G_2\|_{\mathfrak D}^\frac{3}{2})\|\delta F\|_{\mathfrak E'}^2 \\
&\quad \quad + \langle \langle v\rangle^{m_1}Q(\delta G,F_1),\delta F^{(m_1,0,0)}\rangle_{L^2_{x,v}} \\
&\quad \quad + \langle \langle v\rangle^{m_1} \nabla_x \delta \psi \cdot \nabla_v F_1,\delta F^{(m_1,0,0)}\rangle_{L^2_{x,v}}.
\end{align}
Now, again modifying the proof of \eqref{eq:vWeightEnergy}, we have
\begin{align}
&\langle \langle v\rangle^{m_1}Q(\delta G,F_1),\delta F^{(m_1,0,0)}\rangle_{L^2_{x,v}}\\
& \lesssim \|\delta G\|_{\mathfrak E'}\|F_1^{(m_1,0,0)}\|_{L^2_x(L^2\cap \dot{\mathcal H}_\sigma)_v}\|\delta F^{(m_1,0,0)}\|_{L^2_x(L^2\cap\dot{\mathcal H}_\sigma)_v} \\
&\quad +\|\delta G\|_{\mathfrak E'}^\frac{1}{4} \|\delta G\|_{\mathfrak E'\cap \mathfrak D'}^\frac{3}{4} \|F_1^{(m_1,0,0)}\|_{L^2_{x,v}}\|\delta F^{(m_1,0,0)}\|_{L^2_x(L^2\cap \dot{\mathcal H}_\sigma)_v}
\end{align}
Also,
\begin{align}
& \langle \langle v\rangle^{m_1} \nabla_x \delta \psi \cdot \nabla_v F_1,\delta F^{(m_1,0,0)}\rangle_{L^2_{x,v}} \\
 &\quad= -\langle \nabla_x \delta \psi F_1^{(m_1+\frac{3}{2},0)}, \nabla_v \delta F^{(m_1-\frac{3}{2},0)}\rangle_{L^2_{x,v}} \\
 &\quad\quad -(m_1 + \frac{3}{2})\langle  \frac{v}{\langle v\rangle^2} \cdot \nabla_x \delta \psi F_1^{(m_1,0,0)},\delta F^{(m_1,0,0)}\rangle_{L^2_{x,v}}  \\
 &\quad   \lesssim \|\delta\psi\|_{H^{s+1}_x} \| F_1\|_{\mathfrak E} \|(\delta F^{(m_1,0,0)},\nabla_v \delta F^{(m_1-\frac{3}{2},0,0)})\|_{L^2_{x,v}} \\
 &\quad \lesssim  \|\delta\psi\|_{H^{s+1}_x} \| F_1\|_{\mathfrak E} \|\delta F^{(m_1,0,0)}\|_{L^2_x(L^2 \cap \dot{ \mathcal H}_\sigma)_v}.
\end{align}
Altogether, 
\begin{align}
&\frac{d}{dt} \|\delta F^{(m_1,0,0)}\|_{L^2_{x,v}}^2 + \frac{1}{C} \|\delta F^{(m_1,0,0)}\|_{L^2_{x}(\dot{\mathcal H}_\sigma)_v}^2\\
&\quad \leq C\{ (1 + \|G_2\|_{\mathfrak D}^\frac{3}{2} )\|\delta F\|_{\mathfrak E'}^2 +( \|\delta \psi\|_{H^{s+1}_x}^2+ \|\delta G\|_{\mathfrak E'}^2) \|F_1\|_{\mathfrak E\cap \mathfrak D'} ^2 \\
&\quad \quad + \|\delta G\|_{\mathfrak E'}^\frac{1}{2} \|\delta G\|_{ \mathfrak D'}^\frac{3}{2} \|F_1\|_{ \mathfrak E} ^2\}.
\end{align}
Similarly, we also have
\begin{align}
&\frac{d}{dt} \|\delta F^{(m_0,s)}\|_{L^2_{x,v}}^2 + \frac{1}{C} \|\delta F^{(m_0,s)}\|_{L^2_{x}(\dot{\mathcal H}_\sigma)_v}^2\\
&\quad \leq C\{ \langle \|G_2\|_{\mathfrak D}\rangle^\frac{3}{2} \|\delta F\|_{\mathfrak E}^2 +  ( \|\delta\psi\|_{H^{s+1}_x}+ \|\delta G\|_{\mathfrak E'}^2 )\|F_1\|_{\mathfrak E \cap \mathfrak D'}^2 \\
&\quad\quad +  \|\delta G\|_{\mathfrak E'}^\frac{1}{2}\|\delta G\|_{ \mathfrak D'}^\frac{3}{2} \|F_1\|_{\mathfrak E}^2  +  \|\delta F^{(m_1,0,0)}\|_{L^2_{x}(\dot{\mathcal H}_\sigma)_v}^2\}.
\end{align}
Combining the bounds on $F^{(m_1,0,0)}$ and $F^{(m_0,s,0)}$, we have the following: for some choice of $\lambda>0$ sufficiently small depending on $M>0$, we have
\begin{align}
&\frac{d}{dt} \left( \lambda  \|\delta F^{(m_0,s)}\|_{L^2_{x,v}}^2 +   \|\delta F^{(m_1,0,0)}\|_{L^2_{x,v}}^2 \right) +\frac{1}{C} \|\delta F\|_{\mathfrak D'}^2 \\
&\quad \leq C\{ \langle \|G_2\|_{\mathfrak D}\rangle^\frac{3}{2} \|\delta F\|_{\mathfrak E'}^2 +( \|\delta \psi\|_{H^{s+1}_x}^2+ \|\delta G\|_{\mathfrak E'}^2) \|F_1\|_{\mathfrak E\cap \mathfrak D'} ^2 \\
&\quad \quad + \|\delta G\|_{\mathfrak E'}^\frac{1}{2} \|\delta G\|_{ \mathfrak D'}^\frac{3}{2} \|F_1\|_{ \mathfrak E} ^2\}.
\end{align}
Now, $\lambda  \|\delta F^{(m_0,s,0)}\|_{L^2_{x,v}}^2 +   \|\delta F^{(m_1,0,0)}\|_{L^2_{x,v}}^2 \sim \|\delta F\|_{\mathfrak E'}^2$. Hence,
by Gr\"onwall's inequality, we deduce \eqref{eq:contraction}.
\end{proof}

\section{Construction of Solutions}
\label{sec:construction}
\subsection{Solutions to the linear system}
In this section, we use the \textit{a priori} estimates to give a construction scheme for the systems \eqref{eq:VPLunit} and \eqref{eq:VPLion''} using an iteration scheme. Before we can do this, however, we need the following Lemma which guarantees that solutions to the linear system exist (under very strong hypotheses):
\begin{proposition}\label{prop:linearLWP}
\highlightEdit{Fix $T>0$.}
Suppose $F_{in} = F_{in}(x,v)$ satisfies 
\begin{align}
&F_{in} \geq 0\\
&e^{p|v|^2} \nabla_{x,v}^r F_{in} \in L^\infty(\mathbf T^3\times \mathbf R^3), \quad \text{ for all } p\geq 0, r\in \mathbf N_0.
\end{align}
Next, suppose $G = G(t,x,v)$ satisfies 
\begin{align}
&G\geq 0, \label{eq:positivity} \\
&e^{p|v|^2} \nabla_{t,x,v}^r G \in L^\infty([0,T]\times\mathbf T^3\times \mathbf R^3), \quad \text{ for all } p\geq 0, r\in \mathbf N_0. \label{eq:SchwarzG}\end{align}
\highlightEdit{Moreover, let $\psi = \psi(t,x) \in C^\infty([0,T]\times \mathbf T^3)$.}

Then, there exists a unique \highlightEdit{$F = F(t,x,v) \in C^\infty([0,T]\times \mathbf T^3 \times \R^3)$ that solves \eqref{eq:FG}  with initial data $F(t=0) = F_{in}$.  Additionally, $F$ satisfies  the same inequalities as $G$: more precisely, $F(t,x,v)\geq 0$ for all $(t,x,v)\in [0,T]\times \mathbf T^3 \times \mathbf R^3$, and we have $e^{p|v|^2} \nabla_{t,x,v}^r F \in L^\infty([0,T]\times \mathbf T^3 \times \mathbf R^3)$ for any $p\geq 0$ and $r \in \mathbf N_0$. }

\end{proposition}

\begin{proof}
Fix $ p \geq  0$. 
Observe that if such a solution to \eqref{eq:FG} exists, then $U_p = e^{p(|v|^2 +2 \psi)}F$ solves
\begin{align}\label{eq:Up}
(\partial_t -  \partial_t \psi  )U_p  +\{v\cdot \nabla_x -\nabla_x \psi\cdot \nabla_v \} U_p  = e^{p|v|^2} Q(G, e^{-p|v|^2} U_p).
\end{align}
We now construct solutions to the above.
Take $\lambda \in (0,1]$. Consider the following initial value problem:
\begin{equation}
\begin{aligned}\label{eq:Flambda}
&(\partial_t-\partial_t \psi)U_p^{(\lambda)} + \{\frac{v}{\langle \lambda v \rangle}\cdot \nabla_x -\nabla_x \psi\cdot \nabla_v \} U_p^{(\lambda)} \\
&\quad \quad=  e^{p|v|^2} Q(G, e^{-p|v|^2} U_p)+ \lambda\Delta_{x,v}U_p^{(\lambda)},\\
&U_p^{(\lambda)}(t=0) = e^{p(|w|^2 + \psi)} F_{in}.
\end{aligned} 
\end{equation}
Observe that
\begin{align}\label{eq:Qp}
e^{p|v|^2} Q(G, e^{-p|v|^2} U_p^{(\lambda)}) &= \partial_{v_j} (\Phi_{ij} * G \partial_{v_i} U_p^{(\lambda)}- \partial_{v_i} \Phi_{ij} * G U_p^{(\lambda)})\\
&\quad +p(p v_i v_j  -\delta_{ij})\Phi_{ij}* G U_p^{(\lambda)}  - p  v_i \Phi_{ij}*G\partial_{v_j}U_p^{(\lambda)}.
\end{align}
By Lemma \ref{lem:PhiUpperLower}, we see that $\partial_{v_i} \Phi_{ij} * G$, $\Phi_{ij}*G v_i $ and $p(p v_i v_j  -\delta_{ij})\Phi_{ij}* G$ are all bounded. Hence, this a uniformly parabolic PDE with $C^\infty$ coefficients, and thus there exists a unique strong solution  $U_p^{(\lambda)}$ to the above in \highlightEdit{$ \bigcap_{r \geq 0} H^r([0,T]\times \mathbf T^3 \times \mathbf R^3)$.} This fact follows from Chapter 2 of \cite{krylov2008lectures}, with small modifications needed for the domain $\mathbf T^3 \times \mathbf R^3$.
 By applying $\nabla_{x,v}^r$ to both sides, and integrating against $ \nabla_{x,v}^r U_p^{(\lambda)}$, we have for all $m \geq 0$, and $r \in \mathbf N_0$,  
\begin{align}
&\frac{1}{2} \frac{d}{dt} \|\nabla_{x,v}^{r}  U_p^{(\lambda)}\|_{L^2_{x,v}}^2 + \frac{\lambda}{2}  \| \nabla_{x,v}^{r+1} U_p^{(\lambda)}\|_{L^2_{x,v}}^2 \\
&\quad \leq  \langle \nabla_{x,v}^{r}  U_p^{(\lambda)}, \nabla_{x,v}^{r} \{e^{-p|w|^2} Q(G,e^{-p|w|^2}U_p^{(\lambda)}) \}\rangle_{L^2_{x,v}}  \\
&\quad \quad + C\sum_{k = 1}^{r} \|| \nabla_{v}^{k} \left(\frac{v}{\langle \lambda v \rangle}\right)|| \nabla_{x,v}^{r -k+1} U_p^{(\lambda)}|\|_{L^2_{x,v}} \| \nabla_{x,v}^r U_p^{(\lambda)}\|_{L^2_{x,v}}\\
&\quad \quad + C \sum_{k = 1}^{r} \| |\nabla_x^{k+1} \psi| | \nabla_{x,v}^{r -k+1} U_p^{(\lambda)}|\|_{L^2_{x,v}} \|\nabla_{x,v}^r U_p^{(\lambda)}\|_{L^2_{x,v}} \\
&\quad \leq \langle\nabla_{x,v}^{r}  U_p^{(\lambda)},\nabla_{x,v}^{r} \{e^{-p|w|^2} Q(G,e^{p|w|^2}U_p^{(\lambda)})\} \rangle_{L^2_{x,v}}  + C_{\psi,r}  \| U_p^{(\lambda)}\|_{H^k_{x,v}}.
\end{align}
In the above, we use $\nabla_{v}^{k} \left(\frac{v}{\langle \lambda v \rangle}\right) \lesssim_k 1$. Now, using \eqref{eq:Qp},
\begin{align}
&\langle\nabla_{x,v}^{r}  U_p^{(\lambda)},\nabla_{x,v}^{r} \{e^{-p|w|^2} Q(G,e^{p|w|^2}U_p^{(\lambda)})\} \rangle_{L^2_{x,v}}\\
  &\quad =-\langle \nabla_{x,v}^r\{\Phi_{ij} * G \partial_{v_i} U_p^{(\lambda)}\}, \partial_{v_j} \nabla_{x,v}^r U_p^{(\lambda)}\rangle_{L^2_{x,v}} \\
  &\quad\quad +\langle \nabla_{x,v}^r\{\partial_{v_i} \Phi_{ij} * G U_p^{(\lambda)}\}, \partial_{v_j} \nabla_{x,v}^r U_p^{(\lambda)}\rangle_{L^2_{x,v}}\\
&\quad\quad +p\langle \nabla_{x,v}^r\{ (p v_i v_j  -\delta_{ij})\Phi_{ij}* G U_p^{(\lambda)}\},\nabla_{x,v}U_p^{(\lambda)}\rangle_{L^2_{x,v}} \\
&\quad \quad -p \langle \nabla_{x,v}^r \{  v_i \Phi_{ij}*G\partial_{v_j}U_p^{(\lambda)}\},\nabla_{x,v}^r U_p^{(\lambda)}\rangle_{L^2_{x,v}} \\
&\quad \leq -\langle \Phi_{ij} * G \partial_{v_i}\nabla_{x,v}^r U_p^{(\lambda)}, \partial_{v_j} \nabla_{x,v}^r U_p^{(\lambda)}\rangle_{L^2_{x,v}} \\
&\quad \quad + C_{G,p,r} \|  U_p^{(\lambda)}\|_{H^r_{x,v}}^2
\end{align}
Therefore,
\begin{align}
\frac{d}{dt} \|\nabla_{x,v}^{r}  U_p^{(\lambda)}\|_{L^2_{x,v}}^2 \leq C_{G,\psi,p,r} \| \nabla_{x,v}^r U_p^{(\lambda)}\|_{L^2_{x,v}}^2.
\end{align}
 Thus, for all $p \geq 0$, $r \in \mathbf N_0$,
\begin{align} \label{eq:SchwarzFapprox}
\|\nabla_{x,v}^r U_p^{(\lambda)}\|_{L^\infty_t([0,T];L^2_{x,v})} \lesssim_{\highlightEdit{T,F_{in}, G, \psi,p,r}} 1.
\end{align}
Next, for two different $\lambda_1$ and $\lambda_2$, we have that $U_p^{(\lambda_2)} - U_p^{(\lambda_1)}$ solves
\begin{align}
&\partial_t (U_p^{(\lambda_2)} - U_p^{(\lambda_1)}) + \{\frac{v}{\langle \lambda_1 v \rangle}\cdot \nabla_x +\nabla_x \psi\cdot \nabla_v \}(U_p^{(\lambda_2)} - U_p^{(\lambda_1)})\\
&= e^{p|v|^2}Q(G,e^{-p|v|^2}U) + \lambda_1\Delta_{x,v}(U_p^{(\lambda_2)} - U_p^{(\lambda_1)})\\
&\quad + \left(\frac{v}{\langle \lambda_1 v \rangle} - \frac{v}{\langle \lambda_2 v \rangle}\right)\cdot \nabla_x  U_p^{(\lambda_2)} + (\lambda_2 - \lambda_1) \Delta_{x,v} U_p^{(\lambda_2)}.
\end{align}
Now, using \eqref{eq:SchwarzFapprox}, and following similar \textit{a priori} estimates, we have
\begin{align}
\|U_p^{(\lambda_2)} - U_p^{(\lambda_1)}\|_{L^\infty_t([0,T];L^2_{x,v})} \lesssim_{\highlightEdit{T,F_{in},G,\psi,p}} |\lambda_1 -\lambda_2|.
\end{align}
Thus, there exists $U_p \in L^\infty_t([0,T];L^2_{x,v})$ such that  $\lim_{\lambda  \downarrow 0} \|U_p^{(\lambda)} - U_p\|_{L^\infty_t([0,T];L^2_{x,v})} = 0$.
Moreover, we have that  $U_p$ also satisfies \eqref{eq:SchwarzFapprox}. In particular, $U_p$ is a classical solution to \eqref{eq:Up}.  
Next, by a similar estimate, we have that $U_p = e^{p (|v|^2 + 2\psi)} U_0$ for all $p\geq 0$. Setting $F = U_0$, we have that $F$ is a classical solution to \eqref{eq:FG}. Then, for all $p, r$, we have
\begin{align}
\|e^{p|w|^2} F\|_{L^\infty_t([0,T];H^r_{x,v})} \lesssim_{\psi, p,r} \|U_{2p}\|_{{L^\infty_t([0,T];H^r_{x,v})}} \lesssim_{\highlightEdit{T,F_{in},G,\psi, p,r}} 1.
\end{align}
Combining the above estimate with \eqref{eq:FG}, we can show that $\partial_t^k F$ satisfies the same estimate as the above, allowing us to conclude that \eqref{eq:SchwarzG} holds for $F$.

It remains to check that $F \geq 0$. Now, recall that 
\[\nabla_{t,x,v} (\mathbf  1_{\{F < 0\}} F) = \mathbf    1_{\{F< 0\}} \nabla_{t,x,v}  F \text{ a.e.}
\] Thus, we have
\begin{align}
\frac{d}{dt} \|\mathbf  1_{\{F<0\}} F\|_{L^2_{x,v}}^2& = \highlightEdit{2}\langle \partial_t F, \mathbf 1_{\{F<0\}} F\rangle_{L^2_{x,v}}, \\
\langle \{v\cdot \nabla_x - \nabla_x \psi\} F , \mathbf 1_{\{F<0\}} F\rangle_{L^2_{x,v}} & = 0,\\
\langle Q(G,F) ,\mathbf  1_{\{F<0\}} F\rangle_{L^2_{x,v}} &= - \langle \Phi_{ij}  * G\mathbf  1_{\{F<0\}}\partial_{v_i} F, \partial_{v_j}F\rangle_{L^2_{x,v}} \label{eq:indicatorDiffusionTerm1} \\
& \quad + \langle \partial_{v_i}\partial_{v_j} \Phi_{ij}  * G \mathbf 1_{\{F<0\}}F, F\rangle_{L^2_{x,v}}  \label{eq:indicatorDiffusionTerm2}.
\end{align}
\highlightEdit{
Therefore,
\begin{align}
&\frac{1}{2}\frac{d}{dt} \|\mathbf  1_{\{F<0\}} F\|_{L^2_{x,v}}^2 \\
&\quad =- \langle \Phi_{ij}  * G\mathbf  1_{\{F<0\}}\partial_{v_i} F, \partial_{v_j}F\rangle_{L^2_{x,v}}  + \langle \partial_{v_i}\partial_{v_j} \Phi_{ij}  * G \mathbf 1_{\{F<0\}}F, F\rangle_{L^2_{x,v}}.
\end{align}
}
\highlightEdit{Regarding the right hand side of \eqref{eq:indicatorDiffusionTerm1}, we have that $\Phi_{ij}*G(t,x,v) \mathbf 1_{F(t,x,v)<0}$ defines a positive semi-definite matrix for any $(t,x,v)$. Thus, 
\[
 \langle \Phi_{ij}  * G\mathbf  1_{\{F<0\}}\partial_{v_i} F, \partial_{v_j}F\rangle_{L^2_{x,v}} \geq0.
\]
On the other hand, regarding \eqref{eq:indicatorDiffusionTerm2}, we have the identity
\[
\partial_{v_i}\partial_{v_j}\Phi_{ij}*G(t,x,v)  = 8\pi G(t,x,v).
\]
Hence,
\[
\langle Q(G,F) ,\mathbf  1_{\{F<0\}} F\rangle_{L^2_{x,v}} \leq 8\pi \|G\|_{L^\infty} \| 1_{\{F<0\}} F\|_{L^2_{x,v}}^2 \lesssim_G \| \mathbf 1_{F< 0}  F\|_{L^2_{x,v}}^2.
\]
Therefore,
\begin{align}
\frac{d}{dt} \| \mathbf 1_{F< 0}  F\|_{L^2_{x,v}}^2 \leq  C_{G}\| \mathbf 1_{F< 0}   F\|_{L^2_{x,v}}^2. 
\end{align}
We then apply Gr\"onwall's inequality to get 
\begin{align}
0\leq \| \mathbf 1_{F< 0}   F\|_{L^2_{x,v}}^2 \leq C_{\highlightEdit{T,G}} \| \mathbf 1_{F< 0}  F_{in}\|_{L^2_{x,v}}^2 =0.
\end{align}
We conclude $F \geq 0$.}
\end{proof}

\subsection{Construction of solutions to the VPL system} We now have everything we need to prove the second part of the main theorem.\\

\noindent \textit{Proof of Theorem \ref{thm:LWP}-\highlightEdit{(ii)}:} We break the proof into a number of steps---first, construction of solutions for smooth, compactly supported initial data; second, general data in $\mathfrak E$; third, we show continuity in time in $\mathfrak E$; and finally, we show the solutions constructed here are unique. \\

\noindent \textbf{Step 1 (iteration scheme, boundedness):} 
For all $N \geq 1$, fix nonnegative $F_{+,in}^N,F_{-,in}^N\in C^\infty_c(\mathbf T^3 \times\mathbf R^3) $, and $0 < M < \infty$ sufficiently large such that
\begin{align}\label{eq:M/2}
 \|(\frac{1}{n^N_{+,in}},\frac{1}{n^N_{-,in}}) \|_{L^\infty} + \|(F_{+,in}^N,F_{-,in}^N)\|_{\mathfrak E} \leq \frac{M}{2}
\end{align}
Moreover, we assume $\|F_{\pm,in}^N - F_{\pm,in}\|_{\mathfrak E} \to 0$ as $N \to \infty$. In fact, up to taking to a subsequence, we may as well take $\sum_{N\geq 1} \|F_{\pm,in}^N - F_{\pm,in}^{N-1}\|_{\mathfrak E}  < \infty$.

The solution will be given by the following iteration scheme. 
Let 
\[
(F_+^0(t,x,v),F_-^0(t,x,v)) = (\eta(v),\eta(v)),
\]  where $\eta(v)$ is any nonnegative standard $C^\infty$, compactly supported cutoff. We also set $\phi^0(t,x) = 0$.

 Up to taking $M$ larger, we can guarantee that $(F_+^0,F_-^0)$ satisfy the same condition as \eqref{eq:M/2}.
 Then, for each $N \in \mathbf N_0$, define  $(F_+^{N+1},F_-^{N+1}) \in L^2([0,1]\times\mathbf T^3\times\mathbf R^3)$ to be the solution to 
\begin{align}\label{eq:iteration+}
\{\partial_t  +v \cdot \nabla_x  -\nabla_x \phi ^N \cdot \nabla_v\} F_+^{N+1} &= Q(F_+^N + F_-^N, F_+^{N+1}),  \\
\{ \partial_t  + v \cdot \nabla_x  + \nabla_x \phi^N\cdot \nabla_v\}   F_-^{N+1} &=Q(F_+^N+F_-^N, F_-^{N+1}), \label{eq:iteration-}
\end{align}
where $\phi^N = -4\pi\Delta_x^{-1} \int_{\mathbf R^3}F_+^N - F_-^N dv$ with initial data $(F_{+,in},F_{-,in})$. The existence of the iterates is guaranteed by Proposition \ref{prop:linearLWP}. In particular, we have $e^{p|w|^2} \nabla_{t,x,v}^r F_\pm^N \in L^\infty([0,1]\times \mathbf T^3 \times \mathbf R^3)$ for all $r,N\in \mathbf N_0$, $p \geq 0$. We will also denote $\delta F_\pm^N = F_{\pm}^N - F_{\pm}^{N-1}$, and similarly for $\phi^N$, etc. Moreover, for each $N$, we have $F_\pm^N \geq 0$.

Now, for each $N \geq 0$, let $[0,T_N]\subset [0,1]$ be the largest such interval such that
\begin{align}
 \|(\frac{1}{n_{+}^N},\frac{1}{n_{-}^N}) \|_{L^\infty([0,T_N]\times \mathbf T^3)} + \|(F_{+}^N,F_{-}^N)\|_{\mathfrak E_{T_N}} + \|(F_+^N,F_-^N)\|_{\mathfrak D_{T_N}}&\leq M. \label{eq:bootstrapN1}
\end{align}
Let $ \tau \in (0,1]$ also be a small constant to be chosen later. We shall show that $\tau \leq T_{N}$ for all $N \geq 1$. 

We prove this by induction. In the base case $N=0$, this holds trivially, taking $M$ larger if necessary.   
 Now, assume $\tau \leq \inf_{0\leq M \leq N-1} T_M$ for some $N \geq 1$.  In particular, we have $\|\phi^{N-1}(t,\cdot)\|_{H^{s+2}_x}\leq C_M$ for all $t \in [0,\tau]$. Thus, we can apply Proposition \ref{prop:aprioriFG} to get
 \begin{align}
 \|F^N_\pm\|_{\mathfrak E_\tau}  + \frac{1}{C_M} \|F^N_\pm\|_{\mathfrak D_\tau} \leq e^{C_M \sqrt{\tau}}\|F_{\pm,in}^N\|_{\mathfrak E}.
 \end{align}
Since $ \|F_{\pm,in}\|_{\mathfrak E}  \leq \frac{M}{2} $, we may choose $\sqrt{\tau}$ sufficiently small so that
  \begin{align}\label{eq:smallEnergyGrowth}
 \|F^N_\pm\|_{\mathfrak E_\tau}  +  \|F^N_\pm\|_{\mathfrak D_\tau} \leq (1+ C_M\sqrt{\tau})\frac{M}{2}.
  \end{align}
 Moreover, we have the continuity equation
\begin{align}
\partial_t n_\pm^N(t,x)+ \nabla_x \cdot \int_{\mathbf R^3}\highlightEdit{v} F_\pm^N (t,x,v) dv \highlightEdit{=0},
\end{align}
which implies
\begin{align}
\|\partial_t n_\pm^N \|_{L^\infty([0,\tau];H^{s-1})} \lesssim \|F^N_{\pm}\|_{\mathfrak E}. 
\end{align}
Therefore, $\|n_\pm^N(t,x) - n_{\pm,in}(x)\|_{L^\infty([0,\tau]\times\mathbf T^3)} \lesssim_M \tau$. 
In particular, taking $\tau$ small enough,
\begin{align}
 \|(\frac{1}{n_{+}^N},\frac{1}{n_{-}^N}) \|_{L^\infty([0,\tau]\times \mathbf T^3)} \leq (1 + C_M \tau)   \|(\frac{1}{n_{+,in}},\frac{1}{n_{-,in}}) \|_{L^\infty([0,\tau]\times \mathbf T^3)}.
\end{align}
Combining this with \eqref{eq:smallEnergyGrowth}, we see that $\tau$ can be taken small enough, depending only on $M$, such the condition \eqref{eq:bootstrapN1} holds if we replace $T_N$ with $\tau$. Since $T_N$ is the largest number with this property, we have $\tau \leq T_N$.  This completes the inductive step.

In conclusion, there exists some $\tau= \tau(M)$ such that
\begin{align}
 \|(\frac{1}{n_{+}^N},\frac{1}{n_{-}^N}) \|_{L^\infty([0,\tau]\times \mathbf T^3)} + \|(F_{+}^N,F_{-}^N)\|_{\mathfrak E_{\tau}} + \|(F_+^N,F_-^N)\|_{\mathfrak D_{\tau}}&\leq M, \label{eq:bddness}
\end{align}
for all $N \geq 1$.\\

\noindent \textbf{Step 2 (convergence):} For all $N\geq 0$, let $\delta F^N_\pm = F^N_\pm - F^{N-1}$, etc. Then, by Proposition \ref{prop:difference}, and the boundedness proved in previous step, we have for all $N \geq 2$
\begin{align}
\|(\delta F_+^N, \delta F_-^N)\|_{\mathfrak E'_\tau \cap \mathfrak D'_\tau  } &\lesssim_M \|(\delta F_{+,in}^N,\delta F_{-,in}^N)\|_{\mathfrak E'}^2 \\
&\quad +   \int_0^\tau  \|(\delta F_{+}^{N-1},\delta F_{-}^{N-1})\|_{\mathfrak E'}^{\frac{1}{2}}  \|(\delta F_{+}^{N-1},\delta F_{-}^{N-1})\|_{\mathfrak E' \cap \mathfrak D'}^\frac{3}{2} dt \\
&\lesssim_M \|(\delta F_{+,in}^N,\delta F_{-,in}^N)\|_{\mathfrak E'} ^2+ \sqrt{\tau}   \|(\delta F_{+}^{N-1},\delta F_{-}^{N-1})\|_{\mathfrak E'_\tau\cap \mathfrak D'_\tau}^2.
\end{align}
Taking $N' \geq 2$, and summing the above over  all $2 \leq N \leq N'$, we have
\begin{align}
\sum_{1\leq N \leq N'} \|(\delta F_+^N, \delta F_-^N)\|_{\mathfrak E'_\tau \cap \mathfrak D'_\tau  } &\lesssim_M \sum_{1\leq N \leq N'} \|(\delta F_{+,in}^N, \delta F_{-,in}^N)\|_{\mathfrak E'_\tau \cap \mathfrak D'_\tau  } \\
&\quad +  \|(\delta F_{+}^{1},\delta F_{-}^{1})\|_{\mathfrak E'_\tau\cap \mathfrak D'_\tau}\\
&\quad  +\tau^{\frac{1}{4}} \sum_{1\leq N \leq N'} \|(\delta F_+^N, \delta F_-^N)\|_{\mathfrak E'_\tau \cap \mathfrak D'_\tau  }.
\end{align}
Now, by step 1, we know $  \|(\delta F_{+}^{1},\delta F_{-}^{1})\|_{\mathfrak E'_\tau\cap \mathfrak D'_\tau}  \lesssim_M 1$. Moreover, 
\[
\sum_{N = 1}^\infty \|(\delta F_{+,in}^N, \delta F_{-,in}^N)\|_{\mathfrak E'_\tau \cap \mathfrak D'_\tau  } < \infty.\]
 Hence, by taking $\tau$ sufficiently small depending on $M$, we have that 
 \[
 \sum_{N=1}^\infty \|(\delta F_+^N, \delta F_-^N)\|_{\mathfrak E'_\tau \cap \mathfrak D'_\tau  } <\infty. \]
In particular, $(F_+^N,F_-^N)$ are Cauchy in $\mathfrak E'_\tau \cap \mathfrak D_\tau'$. Letting $(F_+,F_-)$ be the limit, we see that $(F_+,F_-) \in C([0,\tau];\mathfrak E')\cap \mathfrak E_\tau \cap \mathfrak D_\tau$, and $n_+,n_-$ enjoy uniform bounds from below.  It is easy to check that $(F_+,F_-)$ satisfy \eqref{eq:VPLunit} in a weak sense.\\

 \noindent \textbf{Step 3 (uniqueness):} The solution constructed by the previous steps is unique as a direct consequence of Proposition \ref{prop:difference}.\\
 
 \noindent \textbf{Step 4 (blow-up criterion):} Since $\tau$ depends only on $M$, we can continue the solution uniquely up until either $\|1/n_\pm(t,\cdot)\|_{L^\infty}$ or $\|F_+(t,\cdot)\|_{\mathfrak E}$ blow-up. \qed

\subsection{Extension to the VPL-Ion system and Landau equation}

The extension to the Landau equation is a trivial modification (and simplification) of the proof given above. On the other hand, the extension to the \eqref{eq:VPLion''}, we must show that solutions Poincar\'e-Poisson system \eqref{eq:PPsystem} below have good boundedness and continuity properties: fixing $G (t,x,v)$, we consider solutions to the system
\begin{equation}
\begin{aligned}\label{eq:PPsystem}
&\frac{d}{dt} \left\{\frac{3}{2\beta(t)} + \iint_{\mathbf T^3 \times \mathbf R^3} \frac{|v|^2}{2} G(t,x,v) dx dv+ \frac{1}{8\pi}  \int_{\mathbf T^3} |\nabla_x \phi (t,x) |^2 dx \right\} = 0,\\
& \beta(0) = \beta_{in},\\
&-\Delta_x  \phi(t,x) = 4\pi ( \int_{\mathbf T^3}G(t,x,v) dv- e^{\beta (t) \phi(t,x)}).
\end{aligned}
\end{equation}

\begin{lemma} \label{lem:PPsystemEstimates} The following statements hold:\\

\noindent(i) Suppose $G \in C([0,T];\mathfrak E')$, $\inf_{x} \int_{\mathbf R^3} G dv >0$, and $\beta_{in} > 0$. Then for some $T >0$, then there exists a unique solution $(\beta, \phi) \in C([0,T^*];\mathbf R_+\times H^{s+2}(\mathbf T^3))$ to \eqref{eq:PPsystem}.  More precisely, let $M >0$ be large enough and $T >0$ small enough so that
\begin{align} \label{eq:GPPbootstrap}
|\ln(\beta_{in})| + \|\frac{1}{\int G dv}\|_{L^\infty([0,T]\times \mathbf T^3)} +\|G\|_{\mathfrak E_T} \leq M
\end{align}
and
\begin{align}\label{eq:kinEnergySmall}
\sup_{t\in[0,T]}\int_{\mathbf T^3 \times \mathbf R^3} \frac{|v|^2}{2} ( G(t,\cdot) - G_{in}  ) dx dv \leq e^{-M}.
\end{align}
Then, $(\beta,\phi)$ exists for all $t \in [0,T]$  and 
\begin{align}\label{eq:aprioriBetaPhi}
\|\ln(\beta)\|_{L^\infty([0,T])} + \|\phi\|_{L^\infty([0,T];H^{s+2})} \lesssim_M 1.
\end{align}\\
(ii) For each $\alpha \in\{1,2\}$ suppose $G^\alpha \in C([0,T];\mathfrak E')\cap \mathfrak E_T$, $\beta_{in}^\alpha \in \{1,2\}$, and let $(\beta^\alpha,\phi^\alpha)$ be the corresponding solution  to \eqref{eq:PPsystem} by setting $G = G^\alpha$ and $\beta_{in} = \beta^\alpha_{in}$. Assume also that each $(\beta^\alpha, G^\alpha)$ satisfy the estimate \eqref{eq:GPPbootstrap}. Letting $\delta G = G^2 - G^1$, $\delta \phi = \phi^2 - \phi^1$ and so on, we have the estimates
\begin{align}\label{eq:PPcontraction}
\|\delta \beta\|_{L^\infty([0,T])} + \|\delta \phi\|_{L^\infty([0,T];H^{s+2})}  \lesssim_M \|\delta G\|_{\mathfrak E'} + |\delta \beta_{in}|.
\end{align}

\noindent (iii)
In addition to the assumptions in part (i),  assume $G$ satisfies \eqref{eq:SchwarzG} in Proposition \ref{prop:linearLWP}. Then $\beta \in C^\infty([0,T])$ and $\phi \in C^\infty([0,T]\times \mathbf T^3)$. 

\end{lemma}

\begin{proof} 

\noindent \textbf{Step 1 (proof of (i), except for time continuity):} We first construct solutions to \eqref{eq:PPsystem} obeying the estimates as in part (i). We shall show $(\beta,\phi) \in L^\infty([0,T];\mathbf R_+ \times H^{s+2}(\mathbf T^3))$ for now, and prove continuity in a later step.

For existence and uniqueness of $(\beta,\phi) \in L^\infty([0,T];\mathbf R_+ \times H^2(\mathbf T^3))$ solving \eqref{eq:PPsystem}, we state (a special case of) part (ii) of Theorem 1.14 in \cite{bardos_maxwellboltzmann_2018}:  for any $0\leq f(x) \in L^2$ and $E > 0$, there exists a unique solution $(\gamma, \psi) \in \mathbf R_+ \times H^2(\mathbf T^3)$ to the system
\begin{align} \label{eq:gamma}
\frac{3}{2\gamma} +\frac{1}{8\pi} \int_{\mathbf T^3} |\nabla_x \psi|^2 dx = E, \\
-\Delta_x \psi =4\pi (f(x)  - e^{\gamma \psi}).\label{eq:psi}
\end{align}
Moreover, for all $\gamma >0$, we can find a unique $\psi \in H^2$ solving only the second line. 
Now, to construct $(\beta ,\phi)$, we first take $\phi_{in} = \psi$ to solve \eqref{eq:psi} with $\gamma = \beta_{in}$.
Now, $\phi_{in}$ exists uniquely in $H^2$.

 In order to solve for $(\beta,\phi)$ for $t >0$, we must ensure that the first line of \eqref{eq:PPsystem}, when integrated on $[0,t]$, gives an admissible condition on $(\beta,\phi)$. Specifically, we take $(\beta,\phi) = (\gamma,\psi)$ as in \eqref{eq:gamma} and \eqref{eq:psi}, with $f(t,x) = \int_{\mathbf R^3} G dv$. The quantity $E$, in accordance with \eqref{eq:PPsystem}, is given by
\begin{align}\label{eq:posCond}
E(t) :=\int_{\mathbf T^3 \times \mathbf R^3} \frac{|v|^2}{2} (G_{in}  - G(t,\cdot)) dx dv+ \frac{3}{2\beta_{in}} +\frac{1}{8\pi} \int_{\mathbf T^3} |\nabla_x \phi_{in}|^2 dx.
\end{align}
Such a solution $(\beta,\phi)$ with $E = \frac{3}{2\beta} + \frac{1}{8\pi} \int_{\mathbf T^3} |\nabla_x\phi|^2dx$ exists as long as $E(t) > 0$. Now, by \eqref{eq:GPPbootstrap} and \eqref{eq:kinEnergySmall},
\begin{align}\label{eq:energyLowerBound}
E(t) \geq -e^{-M} + \frac{3}{2} e^{-M} \geq \frac{1}{2}e^{-M},
\end{align}
guaranteeing the desired condition.


Next, we prove that the assumptions \eqref{eq:GPPbootstrap} and \eqref{eq:kinEnergySmall} imply \eqref{eq:aprioriBetaPhi}. First, we bound $\beta$ from below. Since $E(t) \lesssim_M 1$, indeed $\frac{1}{\beta} \lesssim_M 1$. 

To get an upper bound on $\beta$, we first secure an $L^\infty$ bound on $\phi$. Applying  the maximum principle to the third line of \eqref{eq:PPsystem} we have
\begin{align}
&0\leq \max_x \int_{\mathbf R^3} G dv - e^{ \beta \max_x \phi} ,\\
&0\geq \min_x\int_{\mathbf R^3} G dv - e^{\beta \min_x \phi},
\end{align} and hence
\begin{align}
 \beta \|\phi\|_{L^\infty_x} \leq \|\ln(\int_{\mathbf R^3} G dv)\|_{L^\infty_x} \lesssim_M 1.
\end{align}
If we instead multiply \eqref{eq:PPsystem} by $\phi$, we may apply the above to get
\begin{align}
\frac{1}{4\pi} \|\nabla_x\phi\|_{L^2}^2& = \langle \phi, \int_{\mathbf R^3} G dv-e^{\beta \phi}\rangle_{L^2_x} \\
&\lesssim \|\phi\|_{L^\infty_x} (C_M + e^{\beta    \|\phi\|_{L^\infty_x}}) \\
&\lesssim_M \highlightEdit{\frac{1}{\beta}.}
\end{align}
This implies the upper bound on $\beta$: by \eqref{eq:energyLowerBound},
\begin{align}
\frac{C_M}{\beta} \geq \frac{3}{2\beta} + \frac{1}{8\pi} \|\nabla_x\phi\|_{L^2_x}^2 \geq E \geq \frac{e^{-M}}{2}.
\end{align}
Thus, $\beta \lesssim_M 1$. Combined with the lower bound, we get $|\ln(\beta)|\lesssim_M 1$.

We now control $\|\phi\|_{H^{s+2}_x}$. Applying $|\nabla_x|^s$ to the equation for $\phi_*$, we have
\begin{align}
\frac{1}{4\pi } |\nabla_x |^{s+2} \phi =  |\nabla_x |^s \{ \int_{\mathbf R^3} G dv - e^{ \phi_*}\}.
\end{align}
Using  Theorem 5.2.6 in   \cite{MePise}, 
\begin{align}
\|e^{\beta \phi}\|_{H^{s}_x} \leq C_{\beta, \|\phi\|_{L^\infty_x}} (\|\phi\|_{H^{s}_x} +1) \lesssim_M  \|\phi\|_{H^{s}_x} + 1.
\end{align}
Hence,
\begin{align}
\|\phi\|_{H^{s + 2}_x} \lesssim_M 1+ \|\int_{\mathbf R^3} G dv \|_{H^s_x} + \|\phi\|_{H_x^{s}} \lesssim_M 1 + \|\phi\|_{H^{s+2}_x}^{\frac{s}{s+2}} \|\phi\|_{L^2_x}^{\frac{2}{s+2}}.
\end{align}
By interpolation, we have $\|\phi\|_{H^{s+2}_x} \lesssim_M 1$.\\


\noindent \textbf{Step 2 (control on differences, linearization):} Before proving (ii) and (iii), we first demonstrate how to control differences in the system \eqref{eq:gamma} and \eqref{eq:psi}: first, define the map $\mathscr F : \mathbf R_+ \times H^{s+2}(\mathbf T^3) \to \mathbf R_+ \times H^{s}(\mathbf T^3)$, given by
\begin{align}
\highlightEdit{\mathscr F(\gamma,\psi) = \left( \frac{3}{2\gamma} +\frac{1}{8\pi} \int_{\mathbf T^3} |\nabla_x \psi|^2 dx , \quad 
-\Delta_x \psi +4\pi  e^{\gamma \psi}\right).}
\end{align}
Now, let $d \mathscr F_{(\gamma,\psi)}$ denote the Gateaux derivative of $\mathscr F$ at $(\gamma,\psi)$, i.e. given $(\dot \gamma, \dot \psi) \in \mathbf R \times H^{s+2}(\mathbf T^3)$, define
\begin{align}
\highlightEdit{d \mathscr F_{(\gamma,\psi)}(\dot \gamma,\dot \psi) = \left( -\frac{3\dot \gamma }{2\gamma^2} +\frac{1}{4\pi} \int_{\mathbf T^3} \nabla_x \psi \cdot \nabla_x \dot \psi dx , \quad 
-\Delta_x \dot \psi +4\pi  e^{\gamma \psi} (\dot \gamma \psi + \gamma \dot \psi ) \right).}
\end{align}
Clearly, for $(\gamma,\psi) \in \mathbf R_+ \times H^{s+2}$, this is a bounded linear operator in the sense $dF_{\gamma,\psi} \in \mathcal L(\mathbf R \times H^{s+2}, \mathbf R \times H^{s})$.
We can also show a lower bound on this operator,
\begin{align}\label{eq:gateauxLowerBound}
\|(\dot \gamma, \dot \psi)\|_{\mathbf R \times H^{s+2}} \leq C_{\gamma,\psi} \|d\mathscr F_{(\gamma,\psi)} (\dot \gamma, \dot \psi)\|_{\mathbf R \times H^s}
\end{align}
More precisely, $C_{\gamma,\psi}$ is an increasing function of $|\ln(\gamma)|$ and $\|\psi\|_{H^{s+2}}$. To see this, first set $d \mathscr F_{(\gamma,\psi)}(\dot \gamma,\dot \psi) = (\dot E, \dot f)$. We first solve for $\dot \psi$,
\begin{align}\label{eq:dot_psiSolved}
\dot \psi = (  \gamma e^{\gamma \psi}-\frac{1}{4\pi} \Delta_x)^{-1} (\dot f -e^{\gamma \psi} \psi \dot \gamma ).
\end{align}
Hence, 
\begin{align}
\highlightEdit{\dot E = - \frac{3\dot \gamma}{2\gamma^2} -\frac{1}{4\pi} \int_{\mathbf T^3}\Delta_x  \psi  ( \gamma e^{\gamma \psi}-\frac{1}{4\pi} \Delta_x )^{-1}   (\dot f -e^{\gamma \psi} \psi \dot \gamma )  dx.}
\end{align}
This allows us to solve for $\dot \gamma$:
\begin{align}\label{eq:dotGamIdentity}
\highlightEdit{\dot \gamma  =- \frac{\dot E +\frac{1}{4\pi}\langle \Delta_x \psi,(\gamma e^{\gamma \psi} - \frac{1}{4\pi} \Delta_x)^{-1}\dot f\rangle_{L^2}  }{\frac{3}{2\gamma^2} - \frac{1}{4\pi} \langle \Delta_x\psi, (\gamma e^{\gamma \psi} - \frac{1}{4\pi} \Delta_x)^{-1} (e^{\gamma \psi} \psi)\rangle_{L^2}  }.}
\end{align}
To then bound $|\dot \gamma| \leq C_{\gamma, \psi} \|(\dot E, \dot f)\|_{\mathbf R\times H^s}$, it suffices to show the denominator is bounded from below by some $1/C_\gamma$. This follows by showing the second term in the denominator is nonnegative:
\begin{align}
- \frac{1}{4\pi} \langle \Delta_x\psi, (\gamma e^{\gamma \psi} - \frac{1}{4\pi} \Delta_x)^{-1} (e^{\gamma \psi} \psi)\rangle_{L^2}  \geq 0.
\end{align}
Indeed,
\begin{align}
& -\frac{1}{4\pi} \langle \Delta_x \psi,(\gamma e^{\gamma \psi} -\frac{1}{4\pi} \Delta_x)^{-1}(e^{\gamma \psi} \psi)  \rangle_{L^2_x}  \\
 &=   \langle  (\gamma e^{\gamma \psi} - \frac{1}{4\pi} \Delta_x)\psi,(\gamma e^{\gamma \psi} -\frac{1}{4\pi} \Delta_x)^{-1}(e^{\gamma \psi} \psi)\rangle_{L^2_x} \\
 &\quad  - \gamma \langle e^{\gamma \psi}\psi, (\gamma e^{\gamma \psi} -\frac{1}{4\pi} \Delta_x)^{-1}(e^{\gamma \psi} \psi)\rangle_{L^2_x}\\
 &= \langle  \psi,e^{\gamma \psi}\psi\rangle_{L^2_x} + \gamma \langle e^{\gamma \psi}\psi, (\gamma e^{\gamma \psi} -\frac{1}{4\pi} \Delta_x)^{-1}(e^{\gamma \psi} \psi)\rangle_{L^2_x} \geq 0.
\end{align}
Now that we have bounded $\dot \gamma$, we rely on \eqref{eq:dot_psiSolved} to get $\|\dot \psi\|_{H^{s+2}} \leq C_{\gamma, \psi} \|(\dot E, \dot f)\|_{\mathbf R\times H^s}$. We conclude \eqref{eq:gateauxLowerBound} holds true.

We can use this bound to control finite differences of $\mathscr F$, as well. Consider two $(\gamma_1,\psi_1)$ and $(\gamma_2,\psi_2)$. For each $\lambda \in [1,2]$, define $(\gamma_\lambda,\psi_\lambda) = (\lambda - 1)(\gamma_2,\psi_2)  + (2-\lambda)(\gamma_1,\psi_1)$. Then, by the mean value theorem, there exists some $\rho \in [1,2]$ such that
\begin{align}
\mathscr F(\gamma_2,\psi_2) - \mathscr F(\gamma_1,\psi_1)  = \int_1^2 \frac{d}{d\lambda} \mathscr F(\gamma_\lambda,\psi_\lambda) d\lambda = d\mathscr F_{(\gamma_{\rho},\psi_{\rho})}(\gamma_2-\gamma_1, \psi_2-\psi_1).
\end{align}
Now $|\ln(\gamma_\rho)| \lesssim \sup_{j\in\{1,2\}} |\ln(\gamma_j)|$ and $\|\psi_\rho\|_{H^{s+2}} \leq \sup_{j\in\{1,2\}} \|\psi_j\|_{H^{s+2}}$. Therefore,
\begin{align}\label{eq:FdifferenceBd}
\|(\gamma_2-\gamma_1, \psi_2-\psi_1)\|_{\mathbf R\times H^{s+2}}\leq  (\sup_{j \in \{1,2\}}C_{\gamma_j,\psi_j}) \|\mathscr F(\gamma_2,\psi_2) - \mathscr F(\gamma_1,\psi_1) \|_{\mathbf R\times H^s}.
\end{align}
This estimate will prove useful in the subsequent steps.
\\

\noindent\textbf{Step 3 (time continuity):}  Using the previous step, we complete the proof of (ii). For any two $t_1,t_2 \in [0,T]$, we have by \eqref{eq:FdifferenceBd}
\begin{align}
&\|(\beta(t_2) - \beta(t_1), \phi(t_2,\cdot) - \phi(t_1,\cdot))\|_{\mathbf R\times H^{s+2}} \\
&\quad \lesssim_M \|(E(t_2) - E(t_1), \int_{ \mathbf R^3} G(t_2,x,v) - G(t_1,x,v)dv)\|_{\mathbf R\times H^s}.
\end{align}
Recall \eqref{eq:posCond} for the definition of $E$. Hence,
\begin{align}
\|(\beta(t_2) - \beta(t_1), \phi(t_2,\cdot) - \phi(t_1,\cdot))\|_{\mathbf R\times H^{s+2}}  \lesssim_M \|G(t_2,\cdot) - G(t_1,\cdot)\|_{\mathfrak E}.
\end{align}
Now since $G \in C([0,T];\mathfrak E)$, we conclude that $(\beta,\phi) \in C([0,T];\mathbf R_+ \times H^{s+2})$. This concludes \\

\noindent \textbf{Step 4 (proof of (ii)):} Similarly as in the previous step, we apply \eqref{eq:FdifferenceBd}. For each $t \in [0,T]$,
\begin{align}
\|(\delta \beta(t),\delta \phi(t))\|_{\mathbf R\times H^{s+2}} \lesssim_M \|(\delta E(t),\delta n_+(t,\cdot))\|_{\mathbf R\times H^{s}} 
\end{align}
Here,
\begin{align}
\delta E(t) &=E^2(t) - E^1(t),\\
E^j(t) &= \int_{\mathbf T^3 \times \mathbf R^3} \frac{|v|^2}{2} (G_{in}^j  - G^j(t,\cdot)) dx dv+ \frac{3}{2\beta_{in}^j} +\frac{1}{8\pi} \int_{\mathbf T^3} |\nabla_x \phi_{in}^j|^2 dx.
\end{align}
It is then straightforward to get
\begin{align}
\|(\delta \beta(t),\delta \phi(t))\|_{\mathbf R\times H^{s+2}}  \lesssim \|\delta G\|_{\mathfrak E'} + \|(\delta \beta_{in},\delta \phi_{in})\|_{\mathbf R\times H^{s+2}}
\end{align}
Modifying the proof of \eqref{eq:FdifferenceBd} slightly, we can show that $\|\delta\phi_{in}\|_{H^{s+2}} \lesssim_M  \|\delta G\|_{\mathfrak E'}  + |\delta \beta|$, proving \eqref{eq:PPcontraction}.\\

\noindent \textbf{Step 5 (proof of (iii)):} Since $n_+ \in C^\infty([0,T]\times \mathbf T^3)$, it follows that 
\[
\phi \in L^\infty([0,T];C^\infty(\mathbf T^3))
\]
 easily from the estimates in step 1. 

It suffices to then show that all time derivatives of $(\beta,\phi)$ are bounded.
Let $k \geq 0$. Then
\begin{align}
d\mathscr F_{(\beta,\phi)}(\partial_{t}^k \beta,\partial_{t}^k \phi)= (\partial_{t}^k E,\partial_t^k n_+) + R_k.
\end{align}
where in $R_k$, we have terms involving the time derivatives of $E$, $n_+$, $\beta$ and $\phi$, \textit{up to order} $k -1$.
In particular, $\|R\|_{L^\infty}$ is controlled by some constant depending only on the $C^{k-1}$ norms of $E$, $n_+$, $\beta$ and $\phi$. Note also that all derivatives of $E$ and $n_+$ are both $C^\infty$ on their respective domains.

Modifying step 2 slightly, it is straightforward to show that
\begin{align}
&\|(\partial_{t}^k \beta,\partial_{t}^k \phi)\|_{L^\infty([0,T]) \times L^\infty([0,T]\times \mathbf T^3)} \\
&\quad \leq C_M \|d\mathscr F_{(\beta,\phi)}(\partial_{t}^k \beta,\partial_{t}^k \phi)\|_{L^\infty([0,T]) \times L^\infty([0,T]\times \mathbf T^3)}.
\end{align}
Hence, if $\partial_t^{k-1}\beta$ and $\partial_t^{k-1} \phi$ are both bounded on their respective domains, then so are $\partial_t^{k}\beta$ and $\partial_t^{k} \phi$ . This allows us to induct on $k$, and conclude $\beta$ and $\phi$ are both $C^\infty$ on their respective domains.
\end{proof}

With Lemma \ref{lem:PPsystemEstimates} in hand, we are ready to prove part (ii) of Theorem \ref{thm:LWP}.\\

\noindent \textit{Proof of Theorem 1.1-\highlightEdit{(i)}:} Some details shall be omitted, as the proof is similar to part \highlightEdit{(ii)} in certain aspects.

Once again, we let $F_{+,in}^N\in C^\infty_c(\mathbf T^3 \times\mathbf R^3) $ be a nonnegative approximating sequence for $F_{+,in}$ in $\mathfrak E$. We also let  $F_{+}^0(t,x,v) = \eta (v)$, where $\eta(v)$ is any nonnegative standard $C^\infty$, compactly supported cutoff, $\beta^0 = 1$ and $\phi^0 = 0$.
We also fix $0< M < \infty$ large enough so that 
\begin{align}
|\ln(\beta_{in})| +  \|\frac{1}{n_{+,in}^N}\|_{L^\infty(\mathbf T^3)} + \|F_{+,in}^N\|_{\mathfrak E} \leq \frac{M}{2}.
\end{align}

We now describe the iteration scheme. Suppose  we are given $(F_+^N,\beta^N,\phi^N)$ with $N \in \mathbf N_0$, with $(F_+^N,\phi^N)$ satisfying the same hypotheses as $G$ and $\phi$ in Proposition \ref{prop:linearLWP}, with $t \in [0,1]$, and $\beta \in C^\infty([0,1])$. 
We  then construct $(F^{N+1},\beta^{N+1},\phi^{N+1})$. First, $F_+^{N+1}$ is the solution to 
\begin{align}
(\partial_t + v\cdot \nabla_x -\nabla_x \phi^N)F_+^{N+1} = Q(F_+^N, F_+^{N+1}),
\end{align}
guaranteed by Proposition \ref{prop:linearLWP}. As for $(\beta^{N+1},\phi^{N+1})$, we would like this pair to solve \eqref{eq:PPsystem} with $G = F_+^N$. Unfortunately Lemma \ref{lem:PPsystemEstimates} only guarantees existence for some small time.

 To get around this, 
we  let $T_N \in [0,1]$ be the largest time such that
\begin{align}\label{eq:iteratesBdd}
&\|\ln(\beta^N)\|_{L^\infty([0,T_N])} +  \|\frac{1}{n_+}\|_{L^\infty([0,T_N]\times\mathbf T^3)} + \|F_+^N\|_{\mathfrak E_{T_N}\cap \mathfrak D_{T_N}} \leq M,\\
&\sup_{t\in[0,T]}\int_{\mathbf T^3 \times \mathbf R^3} \frac{|v|^2}{2} ( F_+^N(t,\cdot) - F_{+,in}^N  ) dx dv \leq e^{-M}.
\end{align}
Clearly, these conditions imply that \eqref{eq:GPPbootstrap} and \eqref{eq:kinEnergySmall} both hold with $G = F^N_+$.
 Then, we let $(\beta_*^{N+1},\phi_*^{N+1})$ be the unique smooth solution to \eqref{eq:PPsystem} on the time interval $[0,T_N]$. We then take $0\leq \chi(t)$ to be a smooth cutoff function on $[0,\infty)$, which is identically 1 on $[0,\frac{1}{2}]$, and identically zero on $[1,\infty)$. We then define
\begin{align}
\beta^{N+1}(t) &= \chi(\frac{t}{T_N}) \beta_*^N(t)  + (1-\chi(\frac{t}{T_N}))\beta_{in},\\
\phi^{N+1}(t) &= \chi(\frac{t}{T_N}) \phi_*^{N+1}(t).
\end{align}
so that $(\beta^{N+1},\phi^{N+1})$ are smooth functions defined for all $t \in [0,1]$.

Next, we show boundedness of the iterates on some small time interval $[0,\tau]$ depending only on $M$. We  let $T_N \in [0,1]$ be the largest time such that
\begin{align}
\|\ln(\beta^N)\|_{L^\infty([0,T_N])} +  \|\frac{1}{n_+}\|_{L^\infty([0,T_N]\times\mathbf T^3)} + \|F_+^N\|_{\mathfrak E_{T_N}\cap \mathfrak D_{T_N}} \leq M.
\end{align}
We argue by induction that $\tau < T_N$.  As before, the case $N =0$ is trivial (taking $M$ larger if necessary). Now, assume that for some $N$, we have $\tau < T_{N'}$ for all $0 \leq N' \leq N$.  As in the proof of part (i), we have
\begin{align}
 \|\frac{1}{n_+^N}\|_{L^\infty([0,\tau]\times\mathbf T^3)} + \|F_+^N\|_{\mathfrak E_{\tau}\cap \mathfrak D_{\tau}} \leq \frac{3}{2} \left(\|\frac{1}{n_{+,in}}\|_{L^\infty([0,\tau]\times\mathbf T^3)} + \|F_{+,in}^N\|_{\mathfrak E}\right).
\end{align}
On the other hand, repeating a computation in step 2 of the proof of Lemma \ref{lem:PPsystemEstimates}, in particular the identity \eqref{eq:dotGamIdentity}, we have for all $t \in [0,T_N]$ that 
\begin{align}
&\underline{\dot  \beta}^{N+1} \\
& =\frac{\partial_t\iint_{\mathbf T^3\times\mathbf R^3} \frac{|v|^2}{2} F_+^N dx dv  -\frac{1}{4\pi}\langle \Delta_x \phi_*^{N+1},(\beta_*^{N+1} e^{\beta_*^{N+1} \phi_*^{N+1}} - \frac{1}{4\pi} \Delta_x)^{-1}\partial_t n_+^N\rangle_{L^2}  }{\frac{3}{2(\beta_*^{N+1})^2} - \frac{1}{4\pi} \langle \Delta_x\phi_*^{N+1}, (\beta_*^{N+1} e^{\beta_*^{N+1} \phi_*^{N+1}} - \frac{1}{4\pi} \Delta_x)^{-1} (e^{\beta_*^{N+1} \phi_*^{N+1}} \phi_*^{N+1})\rangle_{L^2}  }.
\end{align}
With the estimates stated Lemma \ref{lem:PPsystemEstimates} in hand, and the positivity of the denominator (also shown in step 2 of the proof of the Lemma), we see
\begin{align}
 \highlightEdit{| \dot \beta_*^{N+1}(t) | \lesssim_M |\iint_{\mathbf T^3\times\mathbf R^3} \frac{|v|^2}{2}\partial_t F_+^N(t,\cdot) dx dv|+ \|\partial_tn_+^N(t,\cdot)\|_{L^\infty(\mathbf T^3)}.}
\end{align}
As shown in the proof of part \highlightEdit{(ii)}, we can control $ \|\partial_tn_+^N(t,\cdot)\|_{L^\infty(\mathbf T^3)} \lesssim_M 1$. As for the kinetic energy, we have either we are in the case $N =0$, where $\partial_t F_+^N = 0$, or 
\begin{align}
\iint_{\mathbf T^3\times\mathbf R^3} \frac{|v|^2}{2}\partial_t F_+^N(t,\cdot) dx dv = \iint_{\mathbf T^3\times \mathbf R^3 }v\cdot \nabla_x \phi^{N-1} F_+^N  + \frac{|v|^2}{2} Q(F_+^{N-1},F_+^N) dx dv.
\end{align}
The first term is easily bounded in terms of some $C_M$. As for the second, we have
\begin{align}
 &\iint_{\mathbf T^3 \times \mathbf R^3} \frac{|v|^2}{2} Q(F_+^{N-1},F_+^N) dxdv \\
 &\quad =  \iint_{\mathbf T^3 \times \mathbf R^3} (\mathrm{tr}(\Phi*F^{N-1})  + 2 v_i \partial_{v_j} \Phi* F^{N-1} )F^N dxdv.
\end{align}
It's then clear that we can bound 
\begin{align}\label{eq:dtEnergy}
|\iint_{\mathbf T^3\times\mathbf R^3} \frac{|v|^2}{2}\partial_t F_+^N(t,\cdot) dx dv| \lesssim_M 1.
\end{align}
 We conclude that $\|{\dot \beta_*}^{N+1}\|_{L^\infty([0,T_N])} \lesssim_M 1$.  
Thus, 
\begin{align}
\|\beta^{N+1}(t)- \beta_{in}\|_{L^\infty([0,\tau])} = \|\chi(\frac{t}{T_N}) (\beta_*^{N+1}(t)- \beta_{in})\|_{L^\infty([0,\tau])} \lesssim_M \tau.
\end{align}
In particular, we can take $\tau$ small enough depending on $M$ so that 
\begin{align}
\|\ln(\beta^{N+1}(t))\|_{L^\infty([0,\tau])} \leq  \frac{3}{2} |\ln(\beta_{in})|.
\end{align}
We thus have that 
\begin{align}
\|\ln(\beta^{N+1}(t))\|_{L^\infty([0,\tau])} + \|\frac{1}{n_+^N}\|_{L^\infty([0,\tau]\times\mathbf T^3)} + \|F_+^N\|_{\mathfrak E_{\tau}\cap \mathfrak D_{\tau}} \leq \frac{3}{2} M.
\end{align}
On the other hand, by \eqref{eq:dtEnergy}, we can take $\tau$ small enough that
\begin{align}
&\sup_{t\in[0,\tau]}\int_{\mathbf T^3 \times \mathbf R^3} \frac{|v|^2}{2} ( F_+^N(t,\cdot) - F_{+,in}^N  ) dx dv \leq C_M\tau \leq \frac{e^{-M}}{2}.
\end{align}
We have thus shown that $\tau < T_{N+1}$, completing the inductive step. Hence, we have uniform boundedness of the iterates $(F_+^{N},\beta^{N},\phi^N)$ on $[0,\tau]$ in the sense of \eqref{eq:iteratesBdd}. 

In a similar fashion as part \highlightEdit{(ii)}, we can argue that these iterates converge to a weak solution to \eqref{eq:VPLion''}. Up to dividing $\tau$ by 2, we have that $(\beta^N,\phi^N) = (\beta_*^N,\phi^N)$ on $t \in [0,\tau]$. Using Proposition \ref{prop:difference} and part (ii) of Lemma \ref{lem:PPsystemEstimates}, we can show that the increments of  $(F_+^N,\beta^N,\phi^N)$ are summable in $\mathfrak E'_\tau\times L^\infty([0,\tau])\times L^\infty([0,\tau]\times H^{s+2})$, and thus Cauchy. Letting $(F_+,\beta,\phi)$ be the limit of this sequence, it is straightforward to show that it defines a weak solution. 
By Proposition \ref{prop:difference} and Lemma \ref{lem:PPsystemEstimates}, this solution is unique as well. 

As $\tau$ depends only on $M$, we can continue the solution $(F_+,\beta,\phi)$ for some maximal interval $[0,T^*)$. As $T \uparrow T^*$, we have either 
\[
\|F_+\|_{\mathfrak E_T} \to \infty, \quad \|1/n_+\|_{L^\infty([0,T]\times\mathbf T^3)} \to \infty, \text{ or } \|\ln(\beta)\|_{L^\infty([0,T])} \to \infty.\] In the third case,
 Clearly $\beta$ is bounded uniformly from below, due to conserved energy involving $\frac{3}{2\beta}$. On the other hand, the conclusion of Corollary 4.4  in \cite{bardos_maxwellboltzmann_2018} holds in the case of our system as well, \highlightEdit{which implies $\beta$ is uniformly bounded from above by a constant depending only on the initial data. 
 In particular, $\beta$ cannot blow-up. }
 
\highlightEdit{ For the sake of a self-contained presentation, we give the proof of the upper bound of $\beta(t)$ as shown in Corollary 4.4 in \cite{bardos_maxwellboltzmann_2018}. By integrating by parts, we have
 \begin{align}
 \frac{1}{2} \frac{d}{dt} \int_{\mathbf T^3} |E|^2 dx &=  -\int_{\mathbf T^3} \phi \Delta_x \phi_t dx\\
 &= 4\pi \int_{\mathbf T^3} \phi\partial_t\{ n_+ - e^{\beta \phi}\} dx\\
 &=  -4\pi \iint_{\mathbf T^3 \times \mathbf R^3}E\cdot v   F_+ dx dv - 4\pi \int_{\mathbf T^3} \phi\partial_t e^{\beta \phi} dx.
 \end{align}
Above, we used the continuity equation,
 \[
 \partial_t n_+ = -\nabla_x \cdot \int_{\mathbf R^3} v  F_+ dv
 \]
 and integrated by parts once more.
On the other hand, the kinetic energy obeys the equation
 \begin{align}
 \frac{d}{dt} \iint_{\mathbf T^3 \times \mathbf R^3} \frac{ |v|^2}{2} F_+dx dv & =- \iint_{\mathbf T^3 \times \mathbf R^3} \frac{|v|^2}{2} E\cdot \nabla_v F_+dx dv  \\
 &=\iint_{\mathbf T^3 \times \mathbf R^3}  E\cdot v F_+dx dv.
 \end{align}
 Combining these two identities, we find that
 \begin{align}
 \frac{d}{dt}\left(\frac{1}{8\pi} \int_{\mathbf T^3} |E|^2 dx +\int_{\mathbf T^3 \times \mathbf R^3} \frac{ |v|^2}{2} F_+dx dv\right) & = - \int_{\mathbf T^3} \phi\partial_t e^{\beta \phi} dx\\
  &= -\frac{1}{\beta} \frac{d}{dt}  \int_{\mathbf T^3} \beta \phi e^{\beta \phi} dx.
 \end{align}
Here, we use the facts that $\beta\phi \partial_t e^{\beta \phi} = \partial_t\{(\beta \phi - 1)e^{\beta \phi}\}$, and $\int_{\mathbf T^3} e^{\beta \phi} dx = 1$. Combining the above with the middle equation of \eqref{eq:VPLion''}, we have
\[
 \frac{d}{dt} \frac{3}{2\beta}  +\frac{1}{\beta} \frac{d}{dt}  \int_{\mathbf T^3} \beta \phi e^{\beta \phi} dx = 0.
\] 
Multiplying each side by $\beta$, we see that this is equivalent to the conservation law
\[
\frac{d}{dt} \left( \ln(\beta) + \frac{2}{3}\int_{\mathbf T^3} \beta \phi e^{\beta \phi} dx\right) = 0.
\]
In particular, $\beta(t) \leq \exp(\ln(\beta_{in}) + \frac{2}{3}\int_{\mathbf T^3} \beta_{in} \phi_{in} e^{\beta_{in} \phi_{in}} dx)$. This demonstrates that $\beta   \in L^\infty([0,T^*))$, as desired.}

\qed

\bibliographystyle{abbrv} 
\bibliography{refs}

\end{document}